\documentclass[12pt]{article}
\usepackage{mathtools} 
\usepackage{amsmath, amssymb} 
\usepackage{amsthm}
\usepackage{enumerate}  
	\usepackage{url} 
	\usepackage{authblk}
	
	\usepackage{tikz}
	\usepackage{tkz-graph}
	\usetikzlibrary{shapes}
	\usetikzlibrary{arrows}
	\usetikzlibrary{decorations.markings, decorations.pathreplacing}
	\usetikzlibrary{positioning}
	
	\usepackage{graphicx}
	\usepackage{caption,subcaption}
	\usepackage{float}
	\usepackage{hyperref}
	
	\newcommand\nn{{\mathbb N}}
	\newcommand\re{{\mathbb R}}
	\newcommand\rats{{\mathbb Q}}

	\DeclarePairedDelimiter\abs{\lvert}{\rvert}%
	\DeclarePairedDelimiter\norm{\lVert}{\rVert}%
	
	\makeatletter
	\let\oldabs\abs
	\def\abs{\@ifstar{\oldabs}{\oldabs*}}
	\let\oldnorm\norm
	\def\norm{\@ifstar{\oldnorm}{\oldnorm*}}
	\makeatother
	
	\newcommand\sbs{\subseteq}

	\newcommand\comp[1]{{\mkern2mu\overline{\mkern-2mu#1}}}
	\newcommand\pmat[1]{\begin{pmatrix} #1 \end{pmatrix}}
	\newcommand\seq[4]{#1_{#2},#1_{#3},\ldots,#1_{#4}}

	\newtheoremstyle{plainsl}%
	{\topsep}
	{\topsep}
	{\slshape} 
	{}
	{\normalfont\bfseries}
	{.}
	{ }
	{}
	
	\swapnumbers
	
	{\theoremstyle{plainsl}
		\newtheorem{theorem}{Theorem}[section]
		\newtheorem{lemma}[theorem]{Lemma}
		\newtheorem{corollary}[theorem]{Corollary}}
	{\theoremstyle{remark}
		}
	
	\renewcommand\proof{\noindent\textsl{Proof. }}
	\newcommand\sqr[2]{{\vbox{\hrule height.#2pt
				\hbox{\vrule width.#2pt height#1pt \kern#1pt
					\vrule width.#2pt}\hrule height.#2pt}}}
	\renewcommand\qed{%
		\ifmmode\eqno\sqr53
		\else\nolinebreak\ \hfill\sqr53\medbreak\fi}

	\DeclareMathOperator{\supp}{supp}
	\DeclareMathOperator{\rk}{rk}
	\DeclareMathOperator{\tr}{tr}
	\DeclareMathOperator{\col}{col}
	
	\DeclareMathOperator{\elsm}{sum}

	\newcommand\ip[2]{\langle#1,#2\rangle}
	\newcommand\one{{\bf1}}

	\newcommand{\proj}[1]{\operatorname{proj}_{#1}}
	\newcommand{\projonto}[2]{\operatorname{proj}_{#1}\!\left(#2\right)}

	\newcommand\AMM{{\widehat{M}}}

	\usepackage{blkarray}

	\title{Laziness of Quantum Walks on Graphs}
	\author[1]{\small Amulya Mohan}
	\author[2]{\small Christino Tamon}
	\author[1]{\small Yichi Xu}
	\author[1]{\small Hanmeng Zhan}
	
	\affil[1]{\small Department of Computer Science, Worcester Polytechnic Institute\\ 
		\texttt{\{amohan,yxu10,hzhan\}@wpi.edu}}
	\affil[2]{\small Department of Computer Science, Clarkson University\\
		\texttt{ctamon@clarkson.edu}}

\begin{document}

		\maketitle
		
\begin{abstract}
The trace of the average mixing matrix of a quantum walk measures the ``laziness" of the walk: the higher the trace, the more likely that the walker returns home in the long run.  In this paper, we develop tools to study this graph invariant arising from Laplacian quantum walks. 

It is known that the complete graph $K_n$ is the laziest connected graph on $n$ vertices. Using our machinery, we show that the star $S_n$ is the second laziest connected graph on $n$ vertices (and hence the laziest tree on $n$ vertices), the complete multipartite graph $K_{n-2,1,1}$ is the third laziest connected graph on $n$ vertices, and the double star $DS(n-3,1)$ is the second laziest tree on $n$ vertices. We also show that on the same number of vertices, more unbalanced double stars are lazier. 
\end{abstract}	
		
\section{Introduction}
Quantum walks provide a universal model of quantum computation \cite{Childs2009}. Over the past few decades, the study of quantum walks has motivated new combinatorial problems, such as the classification of graphs that admit perfect state transfer \cite{Godsil2012,Coutinho2024,Coutinho2015, Alvir2016} and uniform mixing \cite{Godsil2017,Chan2020,Godsil2017a}, and the development of spectral graph theoretic properties of the average mixing matrix \cite{Godsil2013,Coutinho2018,Godsil2018,Godsil2023}.

In this paper, we study a graph invariant arising from Laplacian quantum walks. Let $X$ be a graph with Laplacian matrix $L$. On the $1$-excitation state space, the transition matrix of the quantum walk is
\[U(t) = \exp(itL).\]
Given that the walk starts at vertex $a$, the probability that it is found at vertex $b$ after time $t$ is given by $\abs{U(t)_{ab}}^2$. Thus, the Schur product
\[U(t) \circ \comp{U(t)}\]
records the probability distribution of the quantum walk starting at any vertex; this matrix is called the \textsl{mixing matrix} and denoted $M(t)$. 

While $M(t)$ does not converge as $t$ approaches infinity, its Cesaro limit exists:
	\[\AMM = \lim_{T\to\infty}\frac{1}{T} \int_0^TM(t) dt,\]
and we will call this limit the \textsl{average mixing matrix}. By von Neumann's ergodic theory \cite{vonNeumann1932}, we can express it as
	\[\AMM = \sum_{\lambda} E_{\lambda} \circ E_{\lambda},\]
where $E_{\lambda}$ is the orthogonal projection onto the $\lambda$-eigenspace of $L$. The average mixing matrix is hence symmetric, positive semidefinite, doubly-stochastic with rational entries (see Godsil \cite{Godsil2013}).

The trace of $\AMM$ indicates, on average, how likely the walker returns home. This was first studied by Godsil, Guo and Sobchuk \cite{Godsil2023}. We will develop machinery to further investigate this graph invariant. For simplicity, when $\AMM$ is defined relative to the Laplacian matrix of $X$, we call $\tr(\AMM)$ the \textsl{laziness} of $X$, and denoted it by $\tr(X)$.

For example, using the spectral decomposition of $K_n$:
\[L(K_n)= 0\cdot \frac{1}{n}J + n \left(I - \frac{1}{n} J\right)\]
one obtains the laziness of the complete graph.
\begin{theorem}\cite{Godsil2023}
	The laziness of the complete graph is
	\[\tr(K_n) = n-2 + \frac{2}{n}.\]
\end{theorem}
Godsil, Guo and Sobchuk \cite{Godsil2023} showed that $K_n$ is the laziest connected graph on $n$ vertices.

In this paper, we first establish some basic properties of the laziness, and then show the star $S_n$ is the second laziest connected graph on $n$ vertices (and hence the laziest tree on $n$ vertices). After we find a field trace formula for $\tr(X)$, we show that on the same number of vertices, more unbalanced double stars are lazier. Next, we prove the second laziest tree on $n$ vertices is the most unbalanced double star $DS(n-3,1)$, using a combination of tools including positive semidefinite matrices, eigenvalue interlacing, and properties of the algebraic connectivity. Finally, we extend our study to disconnected graphs, and show that the third laziest connected graph on $n$ vertices is the complete multipartite graph $K_{n-2,1,1}$.

\section{Basic properties}
In this section, we establish some basic properties of the laziness of a graph. The first property is an immediate consequence of the definition of $\tr(X)$.

\begin{theorem}
	Let $X$ and $Y$ be two graphs. Then the laziness of the disjoint union of $X$ and $Y$ is given by
	\[\tr(X\cup Y) = \tr(X) + \tr(Y).\]
\end{theorem}

We will frequently use the following property of positive semidefinite matrices.

\begin{lemma}\cite{Godsil2023}\label{refinement}
Let $E$ and $F$ be two positive semidefinite matrices. Then
\[\tr((E+F)^{\circ 2}) \ge \tr(E^{\circ 2}) + \tr(F^{\circ 2}),\]
with equality if and only if
\[E\circ F=0.\]
\end{lemma}
\proof
Expanding the Schur square gives
\[(E+F)^{\circ 2}= E^{\circ 2} + F^{\circ 2} + 2 E\circ F.\]
The claim now follows since $E\circ F$ is positive semidefinite.
\qed

This is particularly useful when the eigenspaces of a graph refine the eigenspaces of another. For example, Godsil, Guo and Sobchuk \cite{Godsil2023} showed that the laziness of the Cartesian product of two graphs is lower bounded by the product of their laziness.

\begin{theorem}\cite{Godsil2023}\label{cprod}
	Let $X$ and $Y$ be two graphs. Then 
	\[\tr(X\square Y) \ge \tr(X)\tr(Y),\]
	with equality if and only if $\lambda+\theta \ne \lambda'+\theta'$ for any two distinct eigenvalues  $\lambda,\lambda'$ of $L(X)$ and any two distinct eigenvalues $\theta, \theta'$ of $L(Y)$.
\end{theorem}

We extend this observation to direct products of regular graphs.

\begin{theorem}\label{dprod}
	Let $X$ and $Y$ be two regular graphs. Then
	\[\tr(X\times Y) \ge \tr(X) \tr(Y),\]
	with equality if and only if $\lambda \theta \ne \lambda' \theta'$ for any two distinct eigenvalues $\lambda, \lambda'$ of $A(X)$ and any two distinct eigenvalue $\theta, \theta'$ of $A(Y)$.
\end{theorem}
\proof
Suppose $X$ is $k$-regular and $Y$ is $\ell$-regular. Let the spectral decompositions of $A(X)$ and $A(Y)$ be
\[A(X) = \sum_{\lambda} \lambda E_{\lambda},\quad A(Y) = \sum_{\theta} \theta F_{\theta}.\]
Then
\begin{align}
	L(X\times Y) &=D(X) \otimes D(Y) - A(X)\otimes A(Y) \notag \\
	&= k\ell I - A(X)\otimes A(Y)  \notag \\
	&= \sum_{\lambda}\sum_{\theta} (k\ell - \lambda \theta) E_{\lambda}\otimes F_{\theta}, \label{directproduct}
\end{align}
where \eqref{directproduct} is a refinement of the spectral decomposition of $L(X\times Y)$. Therefore by Lemma \ref{refinement}, we have 
\begin{align*}
	\tr(X\times Y)&\ge \sum_{\lambda} \sum_{\theta} \tr(E_{\lambda}^{\circ 2}\otimes F_{\theta}^{\circ 2})\\
	&=\sum_{\lambda} \sum_{\theta}  \tr(E_{\lambda}^{\circ 2}) \tr(F_{\theta}^{\circ 2})\\
	&=\sum_{\lambda}  \tr(E_{\lambda}^{\circ 2})\sum_{\theta}  \tr(F_{\theta}^{\circ 2})\\
	&=\tr(X) \tr(Y). 
\end{align*}
Moreover, equality holds if and only if \eqref{directproduct} is the spectral decomposition of $L(X\times Y)$; that is, 
\[k\ell -\lambda \theta \ne k\ell - \lambda' \theta'\]
for any two distinct eigenvalues $\lambda, \lambda'$ of $A(X)$ and any two distinct eigenvalue $\theta, \theta'$ of $A(Y)$.
\qed

Our next result shows that the laziness of a graph $X$ is related to the laziness of its complement $\comp{X}$, and their difference depends only on the number of vertices in $X$, the number of components in $X$, and the number of components in $\comp{X}$.

\begin{lemma}\label{complement}
	Let $X$ be a graph on $n$ vertices with $c$ components. Suppose $\comp{X}$ has $\comp{c}$ components. Then
	\[\tr(X) - \tr(\comp{X}) = \frac{2}{n} (c-\comp{c}).\]
\end{lemma}
\proof 
Let the spectral decomposition of $L(X)$ be
\[L(X) = \sum_{\lambda} \lambda E_{\lambda}.\]
Let 
\[E = E_0 - \frac{1}{n}J,\]
and let  $F$ be the orthogonal projection onto $\ker(nI - L(X))$. It is a well-known result (see, for example, \cite[Section 13.1]{Godsil2001}) that 
\[\rk(E) = c-1,\quad \rk(F) = \comp{c}-1.\]
Moreover, every eigenvector for $L(X)$ is also an eigenvector for $L(\comp{X})$. Thus we have
\begin{align*}
	L(X)&=0\cdot \left(\frac{1}{n}J+ E\right) + \sum_{0<\lambda<n} \lambda E_{\lambda} + n  F\\
	L(\comp{X})&=0\cdot \left(\frac{1}{n}J + F\right)+ \sum_{0<\lambda<n} (n-\lambda) E_{\lambda} + n E
\end{align*}
Therefore
\[\AMM(X)-\AMM(\comp{X}) = \left(\frac{1}{n}J+ E\right)^{\circ 2} +F^{\circ 2}- \left(\frac{1}{n}J + F\right)^{\circ 2} - E^{\circ 2} =\frac{2}{n}(E-F),\]
from which the claim follows.
\qed

Let $S_n$ be the star on $n$ vertices. Applying the above result to its complement yields the laziness of $S_n$.

\begin{theorem}
	The laziness of the star is
	\[\tr(S_n) = n-2 + \frac{2}{n-1} - \frac{2}{n}.\]
\end{theorem}
\proof
Using the complement formula and the union formula, we obtain
		\begin{align*}
	\tr(S_n) &= \tr(K_1 \cup K_{n-1}) -\frac{2}{n}\\
	&=\tr(K_1) + \tr(K_{n-1}) - \frac{2}{n}\\
	&=n-2+\frac{2}{n-1}-\frac{2}{n} \tag*{\sqr53}
\end{align*}

The \textsl{join} of two graphs $X$ and $Y$, denoted $X + Y$, is obtained from the disjoint union $X\cup Y$ by joining each vertex of $X$ to each vertex of $Y$. Combining the complement formula and the union formula gives the following.

\begin{corollary}\label{join}
	Let $X$ be a graph on $n$ vertices with $c$ components, and suppose $\comp{X}$ has $\comp{c}$ components. Let $Y$ be a graph on $m$ vertices with $d$ components, and suppose $\comp{Y}$ has $\comp{d}$ components. Then
	\[\tr(X+Y) = \tr(X) + \tr(Y) + \frac{2(\comp{c}-c)}{n} + \frac{2(\comp{d}-d)}{m} + \frac{2(1-\comp{c}-\comp{d})}{n+m}.\]
\end{corollary}
\proof
This follows from the observation
\[X + Y = \comp{\comp{X} \cup \comp{Y}}. \tag*{\sqr53}\]

As an application of these properties, we show that there exist families of graphs with any desired asymptotic laziness. We will build them from two families of graphs---the paths and the cycles---with asymptotically constant laziness. Their average mixing matrices were found by Godsil \cite{Godsil2013}.

\begin{theorem}\cite{Godsil2013}
	Let $T$ be the $01$-matrix with ones on the anti-diagonal. The average mixing matrix of $P_n$ is
	\[\AMM (P_n)= \frac{1}{n^2}\left((n-1)J + \frac{n}{2}(I+T)\right).\]
\end{theorem}

\begin{theorem}
	The laziness of the path is
	\[\tr(P_n) = \begin{cases}
		\frac{3}{2}-\frac{1}{2n}, & \text{ if $n$ is odd},\\
		\frac{3}{2} - \frac{1}{n}, & \text{ if $n$ is even}.
	\end{cases}\]
\end{theorem}

\begin{theorem}\cite{Godsil2013}
	Let $P$ be the permutation matrix corresponding to a cycle of length $n$. The average mixing matrix of $C_n$ is
	\[\AMM(C_n) = \begin{cases}
		\frac{n-1}{n^2}J + \frac{1}{n}I, & \text{ if $n$ is odd},\\
		\frac{n-2}{n^2}J + \frac{1}{n}(I + P^{n/2}), & \text{ if $n$ is even}.
	\end{cases}\]
\end{theorem}

\begin{corollary}
	The laziness of the cycle $C_n$ is
	\[\tr(C_n)=\begin{cases}
		2-\frac{1}{n}, & \text{ if $n$ is odd},\\
		2- \frac{2}{n}, & \text{ if $n$ is even}.
	\end{cases}\]
\end{corollary}

We now construct three families of connected graphs with asymptotic laziness $\Theta(f(n))$, using joins, Cartesian products and direct products, respectively.

\begin{corollary}
	Let $f:\nn \to \re$ be a function satisfying $1\le f(n)\le n$. 
	\begin{enumerate}[(i)]
		\item Let $n\ge 2$. Let $X_n = \comp{K}_{g(n)} + \comp{P_{h(n)}}$, where
		\[g(n) = \min\{\lceil f(n) \rceil, n-1\},\quad h(n) =n- g(n).\]
		Then $X_n$ is a connected graph on $n$ vertices with
		\[\tr(X_n) = \Theta(f(n)).\]
		\item Let $n\ge 3$. Let $X_n = K_{g(n)} \square C_n$, where
		\[g(n) = \max\{ \lceil f(n) \rceil, 5\}.\]
		Then $X_n$ is a connected graph on $ng(n)$ vertices with
		\[\tr(X_n) = \Theta(f(n)).\]
		\item Let $n\ge 1$. Let $X_n = K_{g(n)} \times C_{2n+1}$, where
		\[g(n) =\max\left\{ \lceil f(n) \rceil, 4\right\}.\]
		Then $X_n$ is a connected graph on $(2n+1) g(n)$ vertices with
		\[\tr(X_n) = \Theta(f(n)).\]
	\end{enumerate}
\end{corollary}
\proof
For the first construction, apply Lemma \ref{complement}:
\[\tr(X_n)=\tr(\comp{K_{g(n)} \cup P_{h(n)}})  = \tr(K_{g(n)}) + \left(\tr(P_{h(n)}) - \frac{2}{n}\right).\]
By the laziness of $K_{g(n)}$ and $P_{h(n)}$, we have
\[	\frac{g(n)}{2}\le \tr(K_g(n)) \le g(n),\quad 0\le \tr(P_{h(n)}) - \frac{2}{n}\le \frac{3}{2}.\]
Hence 
\[\tr(X_n) = \Theta(g(n)) = \Theta(f(n)).\]

For the second construction, note that for $g(n)\ge 5$, every eigenvalue $\theta$ of $L(C_n)$ satisfies
\[0\le \theta\le 4,\quad g(n)+\theta\ge 5,\]
and so by Theorem \ref{cprod},
\[\tr(X_n) =\tr(K_{g(n)})\tr(C_n).\]
By the laziness of $K_{g(n)}$ and $C_n$, we have
\[	\frac{g(n)}{2}\le \tr(K_g(n)) \le g(n),\quad 1\le \tr(C_n)\le 2.\]
Hence
\[\tr(X_n) = \Theta(g(n)) = \Theta(f(n)).\]

For the third construction, note that for $g(n)\ge 4$, no two distinct eigenvalues $\theta$ and $\theta'$ of $A(C_{2n+1})$ satisfy
\[(g(n)-1)\theta = -\theta',\]
and so by Theorem \ref{dprod},
\[\tr(X_n) = \tr(K_{g(n)})\tr(C_{2n+1}).\]
By the laziness of $K_{g(n)}$ and $C_{2n+1}$, we have
\[	\frac{g(n)}{2}\le \tr(K_g(n)) \le g(n),\quad 1\le \tr(C_{2n+1})\le 2.\]
Hence 
\[\tr(X_n) = \Theta(g(n)) = \Theta(f(n)). \tag*{\sqr53}\]

\section{Laziest and second laziest connected graphs}

It was shown by Godsil, Guo and Sobchuk \cite{Godsil2023} that the laziest connected graph on $n$ vertices is the complete graph.

\begin{theorem}\cite{Godsil2023}
	Let $X$ be a connected graph on $n$ vertices. Then
	\[\tr(X) \le \tr(K_n),\]
	with equality if and only if $X=K_n$.
\end{theorem}

In this section, we prove that the second laziest connected graph on $n$ vertices is $S_n$.

\begin{theorem}\label{star}
	Let $X$ be a connected non-complete graph on $n$ vertices. Then
	\[\tr(X) \le \tr(S_n),\]
	with equality if and only if $X=S_n$.
\end{theorem}
\proof
Since $X\ne K_n$, its Laplacian matrix has at least three distinct eigenvalues. Suppose first that some eigenvalue, say $\lambda$, has multiplicity $k$ with $2\le k\le n-3$. Note that 
\[\max\{\tr(E_{\lambda}^{\circ 2}): \tr(E_{\lambda}) = k\}\]
is attained when $k$ diagonal entries of $E_{\lambda}$ is $1-\frac{1}{n}$, one diagonal entry of $E_{\lambda}$ is $\frac{k}{n}$, and the remaining diagonal entries of $E_{\lambda}$ are $0$. Thus by Lemma \ref{refinement}, we have
\begin{align*}
	\tr(X) &\le \tr(E_0^{\circ 2}) + \tr(E_{\lambda}^{\circ 2}) + \tr((I-E_0-E_{\lambda})^{\circ 2})\\
	&= 2\tr(E_{\lambda}^{\circ 2}) -\frac{2(n-1)}{n} \tr(E_{\lambda}) + \frac{(n-1)^2+1}{n}\\
	&\le \frac{(n-1)^2+1}{n}-\frac{2k(n-1-k)}{n^2} \\
	&< n-2+\frac{2}{n-1}-\frac{2}{n}\\
	&=\tr(S_n). 
\end{align*}
Now suppose no such eigenvalue exists. Then $X$ has exactly three distinct eigenvalues, 
\[0^{(1)},\quad \mu^{(n-2)},\quad \lambda^{(1)}.\]
Write 
\[E_{\lambda} = xx^T\]
for some real unit vector $x$. Note that
\[
\max \left\{\sum_i x_i^4: \sum_i x_i^2 =1, \sum_i x_i=0\right\}
\]
is attained at 
\begin{equation}\label{startight}
	x = \frac{1}{\sqrt{n(n-1)}}\pmat{ n-1 & -1 & \cdots & -1}^T.
\end{equation}
Hence
\begin{align*}
	\tr(X) &= \tr(E_0^{\circ 2}) + \tr(E_{\lambda}^{\circ 2}) + \tr((I - E_0 - E_{\lambda})^{\circ 2})\\
	&= 2\tr(E_{\lambda}^{\circ 2}) -\frac{2(n-1)}{n} \tr(E_{\lambda}) + \frac{(n-1)^2+1}{n}\\
	&\le n-2 + \frac{2}{n-1} -\frac{2}{n}\\
	&=\tr(S_n).
\end{align*}
Finally, suppose the above inequality is tight. Then \eqref{startight} is an eigenvector for $\lambda$. Let $u$ be the vertex  attaining the largest entry in $x$. By $Lx=\lambda x$ we deduce
\[\deg(u) = n-1,\quad \lambda =n.\]
Thus $\comp{X} = K_1 \cup Y$ where $Y$ has exactly two distinct eigenvalues. It follows that $Y=K_{n-1}$, and so $X=S_n$.
\qed

This immediately implies that $S_n$ is the laziest tree on $n$ vertices.

\begin{corollary}
	Let $T$ be a tree on $n$ vertices. Then
	\[\tr(T)\le \tr(S_n),\]
	with equality if and only if $T=S_n$.
\end{corollary}
		
\section{A field trace formula}
In this section, we derive a formula for $\tr(X)$ using certain characteristic polynomials. We then apply this formula to find the laziness of double stars.

Given a graph with Laplacian matrix $L$, let $\phi(L, x)=\det(xI-L)$ be the characteristic polynomial of $L$. Let $L\backslash  u$ denote the matrix obtained from $L$ by deleting its $u$-th row and $u$-th column. For each vertex $u$, let $\phi_u(L, x) = \det(xI-L\backslash u)$. The following result is the Laplacian analogue of Equation (3) in \cite[Section 4.4]{Godsil1993}. We use the standard notation $f^{(m)}(x)$ for the $m$-th derivative of a function $f(x)$.

\begin{lemma}
	Let $\lambda$ be a Laplacian eigenvalue of $X$ with eigenprojection $E_{\lambda}$.  If $\lambda$ has multiplicity $m$, then for any vertex $u$, we have
	\[(E_{\lambda})_{u,u} = m \frac{\phi_u^{(m-1)}(L, \lambda)}{\phi^{(m)}(L,\lambda)}.\]
\end{lemma}
\proof
Write 
\[\phi(L, x) = (x-\lambda)^m g(x)\]
for some polynomial $g(x)$ coprime to $x-\lambda$. Since the eigenvalues of $L\backslash u$ interlace those of $L$, we may write
\[\phi_u(L, x) = (x-\lambda)^{m-1} h(x)\]
for some polynomial $h(x)$. By the spectral decomposition of $L$ and Cramer's rule,
\[  \sum_{\lambda} \frac{(E_{\lambda})_{uu}}{x-\lambda} =(xI-L)^{-1}_{uu}= \frac{\phi_u(L, x)}{\phi(L, x)}.\]
Thus
\begin{align*}
	(E_{\lambda})_{uu} 
	&= \lim_{x\to \lambda} \frac{\phi_u(L,x)(x-\lambda)}{\phi(L, x)}\\
	&=\lim_{x\to \lambda} \frac{(x-\lambda)^{m-1} h(x) (x-\lambda)}{(x-\lambda)^m g(x)}\\
	&=\lim_{x\to \lambda} \frac{h(x)}{g(x)}.
\end{align*}
The result follows from the fact that
\[\phi^{(m)}(L, \lambda) = m! g(\lambda)\]
and 
\[\phi_u^{(m-1)}(L,\lambda) = (m-1)! h(\lambda).\tag*{\sqr53}\]

This yields a field trace formula for $\tr(X)$. Given two polynomials $f(x), g(x)$ in  $\rats[x]$ with $f\ne 0$, the trace $\tr^{\rats[x]/(f)}_{\rats}g(x) $ is the trace of the linear map $\rats[x]/(f) \to \rats[x]/(f)$ given by:
\[h(x) \mapsto h(x) g(x) \pmod{f(x)}.\]

\begin{theorem}
	Let $X$ be a graph with Laplacian matrix $L$. Suppose the characteristic polynomial of $L$ factors into 
	\[\phi(L,x) = f_1(x)^{m_1} \cdots f_k(x)^{m_k},\]
	where each $f_i$ is irreducible over $\rats$, and $\gcd(f_i, f_j)=1$ whenever $i\ne j$. For each vertex $u$, let
	\[\psi_{i,u}(L, x) = m_i \frac{\phi_u^{(m_i-1)}(L, x)}{\phi^{(m_i)}(L, x)}.\]
	Then
	\[\tr(X) =  \sum_{i=1}^k \sum_u\tr^{\rats[x]/ (f_i)}_{\rats} \psi_{i, u}^2(L, x).\]
\end{theorem}
\proof
For any root $\lambda$ of $f_i(x)$, we have
\[\sum_{\lambda: f_i(\lambda)=0} (E_{\lambda})^2_{u,u} = \sum_{\lambda: f_i(\lambda)=0} \psi_{i,u}^2(\lambda)=\tr^{\rats(\lambda)}_{\rats} \psi_{i,u}^2(\lambda).\]
By the isomorphism from $\rats(\lambda)$ to $Q[x]/(f_i)$:
\[\psi_{i,u}(\lambda) \mapsto \psi_{i,u}(x) \pmod{f_i(x)},\]
we obtain
\[\tr(X) =  \sum_{i=1}^k \sum_u\tr^{\rats[x]/ (f_i)}_{\rats} \psi_{i, u}^2(x). \tag*{\sqr53}\]

A partition $\pi=\{\seq{C}{1}{2}{d}\}$ of $V(X)$ is \textsl{equitable} if for any $i$ and $j$, there is a constant $\gamma_{ij}$ such that every vertex in $C_i$ has $\gamma_{ij}$ neighbors in $C_j$. Given an equitable partition of $V(X)$ with characteristic matrix $P$, there is a $Q$, called the \textsl{quotient matrix}, such that
\[LP = PQ.\]
It is a well-known result (see, for example,  \cite[Section 9.3]{Godsil2001}) that the characteristic polynomial of $Q$ divides that of $L$. We extend our notation $\phi(L, x)$ and $\phi_u(L,x)$ to 
\[\phi(Q,x) = \det(xI-Q),\quad \phi_i(Q, x) = \det(xI - Q\backslash i),\]
 and show how to compute part of $\tr(X)$ using the latter polynomials. 

\begin{lemma}\label{eqtrace}
	Let $X$ be a graph with an equitable partition $\{C_1,\ldots,C_d\}$. Let $Q$ be the quotient matrix, and let $c_j=\abs{C_j}$ for $j=1,2,\cdots,d$. Let $f(x)\in \rats[x]$ be an irreducible factor of $\phi(Q, x)$ of multiplicity $m$, and suppose 
	\[\gcd\left(f(x), \frac{\phi(L, x)}{\phi(Q, x)}\right) = 1.\]
	For each cell $C_j$, define
	 \[
	\psi_j(Q, x)
	=m\frac{\phi_j^{(m-1)}(Q,x)}{\phi^{(m)}(Q,x)}.
	\]
	Then the contribution to $\tr(X)$ from the roots of $f(x)$ is
\[
\sum_{j=1}^d\frac{1}{c_j}
\tr_{\mathbb{Q}}^{\mathbb{Q}[x]/(f)}
\psi_j^2(Q, x).
\]
\end{lemma}
\proof 
Let $\widehat{P}$ be the normalized characteristic matrix, that is,
\[\widehat{P} = P(P^TP)^{-1/2}.\]
Note that $Q$ is similar to the symmetric matrix
\[\widehat{Q} = \widehat{P}^T L \widehat{P}.\]
For each eigenvalue $\lambda$ of $Q$, let $F_{\lambda}$ be the projection onto the $\lambda$-eigenspace of $\widehat{Q}$, and let $E_{\lambda}$ be the projection onto the $\lambda$-eigenspace of $L$. Since $f(x)$ is coprime to $\phi(L,x)/\phi(Q,x)$, we have
\[E_{\lambda} = \widehat{P} F_{\lambda} \widehat{P}^T.\]
Therefore for any $u\in C_j$,
\[(E_{\lambda})_{uu} = \frac{1}{c_j} (F_{\lambda} )_{jj} = \frac{1}{c_j} \psi_j(Q,\lambda). \tag*{\sqr53}\]

\section{Double stars}
We apply results from the earlier section to find the laziness of double stars. Let $DS(\ell, r)$ denote the double star on $n$ vertices with $\ell$ leaves on one side and $r$ leaves on the other side, as shown in Figure \ref{DS}.

\begin{figure}[H]
	\tikzset{every node/.style={circle}}
	\begin{center}
		\begin{tikzpicture}
			\node[draw] at (0,0.25) (0) {};
			\node[draw] at (2,0.25) (1) {};
			
			\node[draw] at (4,1.5) (k) {$1$};
			\node[draw] at (4,0.5) (kk) {$2$};
			\node[draw,draw=white] at (4,-0.25) (dots) {$\vdots$};
			\node[draw] at (4,-1) (n) {$r$};
			
			\node[draw] at (-2.0,1.5) (3) {$1$};
			\node[draw] at (-2.0,0.5) (4) {$2$};
			\node[draw,draw=white] at (-2.0,-0.25) (dots) {$\vdots$};
			\node[draw] at (-2.0,-1) (l) {$\ell$};
			\draw (0) -- (1);
			\draw (0) -- (3);
			\draw (0) -- (4);
			\draw (0) -- (l);
			\draw (1) -- (k);
			\draw (1) -- (kk);
			\draw (1) -- (n);
		\end{tikzpicture}
	\end{center}
	\caption{$DS(\ell,r)$}
	\label{DS}
\end{figure}
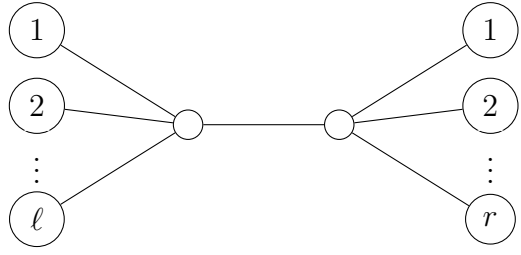

\begin{theorem}
	We have
	\[\tr(DS(\ell, r)) = \frac{1}{n} + n-6 + \frac{f(n,s)}{nsg(n,s)},\] 
	where $s=\ell r$ and
	\begin{align*}
		f(n,s)
		=&4(2n-3)s^3
		-\left(2n^3-43n^2+88n+24\right)s^2 \\
		&-4\left(2n^4-18n^3+31n^2-9n+3\right)s 
		-6n(n-2)(n-1)^3,\\
		g(n,s)
		=&4s^2-\left(n^2-20n-8\right)s-4(n-1)^3.
	\end{align*}
\end{theorem}
\proof 
The double star $DS(\ell, r)$ has an equitable partition with quotient matrix 
\[Q = \pmat{1&-1&0&0\\
	-\ell & \ell+1 & -1 & 0\\
	0& -1 & r+1 & -r\\
	0& 0& -1 & 1}.\]
Hence the eigenvalues of $DS(\ell, r)$ are
\[0^{(1)},\quad 1^{(n-4)},\quad \lambda^{(1)}, \quad \mu^{(1)}, \quad \theta^{(1)},\]
where $\lambda$, $\mu$, $\theta$ are the roots of
\begin{equation}\label{cubic}
	\frac{\phi(Q, x)}{x} = x^3 - (n+2) x^2 + (\ell r + 2n + 1)x - n.
\end{equation}
Since the eigenprojections onto the $0$-eigenspace and $1$-eigenspace are, respectively,
\[E_0 = \frac{1}{n}J,\quad E_1 = \pmat{
	I - \frac{1}{\ell} J & &  \\
	& 0 &\\
	& & I - \frac{1}{r}J}.\]
their contributions to $\tr(DS(\ell, r))$ are
\begin{equation}\label{E0E1}
	\tr(E_0^{\circ 2})=\frac{1}{n},\quad \tr(E_1^{\circ 2}) = n-6 + \frac{n-2}{\ell r}.
\end{equation}
To compute the remaining contribution, note that
\begin{align*}
	\phi_1(Q, x) &= x^3 - (n+1) x^2 + ((r+1)\ell +n) x -\ell \\
	\phi_2(Q, x) &= x^3 - (r+3)x^2 + (r+3) x - 1\\
	\phi_3(Q, x) &= x^3 - (\ell+3) x^2 + (\ell+3) x-1\\
	\phi_4(Q, x) &= x^3 - (n+1)x^2 + (r(\ell+1) + n) x - r
\end{align*}
If $\ell \ne  r$, then \eqref{cubic} is irreducible. Otherwise, it factors into two irreducible polynomials:
\[\frac{\phi(Q,x)}{x-1}=(x-(s+1)) (x^2-(s+3)x + 2),\]
where $s=\ell =r$. 
In both cases, a calculation using Lemma \ref{eqtrace} and SageMath shows \eqref{cubic} contributes
\[\frac{f(n,s)}{npg(n,s)} - \frac{n-2}{s}.\] 
to $\tr(DS(\ell, r))$.
\qed

In particular, we get the following formula and lower bound for the laziness of $DS(n-3,1)$.
\begin{theorem}\label{dsn-31}
	The laziness of $DS(n-3,1)$ is
	\[\tr(DS(n-3,1)) = n-\frac{14}{5} - \frac{4(14n^4-165n^3+705n^2-1108n+120)}{5n(n-3)(5n^3-39n^2+88n-16)}.\]
	If $n\ge 5$, then
	\[\tr(DS(n-3,1))> n-\frac{14}{5} - \frac{56}{25n}.\]
\end{theorem}

Using SageMath, we find that the derivative of $\tr(DS(\ell, r))$ with respect to $s=\ell r$ is negative. It follows that more unbalanced double stars are lazier. 

\begin{theorem}
	Suppose $1\le \ell\le \frac{n}{2}-2$. Then
	\[\tr(DS(\ell, r))>\tr(DS(\ell+1, r-1)).\]
\end{theorem}
\proof
The derivative of $\tr(DS(\ell, r))$ with respect to $s=\ell r$ has the same sign as
\begin{align*}
	h(n,s)=&44s^4-8(9n^2-16n-2)s^3\\
	&+\left(27n^4-242n^3+276n^2+174n+8\right)s^2\\
	&-3(n-2)(n-1)^3(n^2-20n-8)s\\
	&-6(n-2)(n-1)^6.
\end{align*}
For $1\le r \le \frac{n}{2}-1$, the product $s=\ell r$ is strictly increasing, and we have 
\[n-3\le s \le \frac{(n-2)^2}{4}.\]
Thus we may write  
\[s = m+1+ \frac{m^2}{4} t\]
for some $m=n-4$ and $0\le t\le 1$. Substituting this into $h(n,s)$ gives
\begin{align*}
	64\,h\left(m+4,m+1+\frac{m^2}{4}t\right)
	=&b_0(1-t)^4+4b_1t(1-t)^3+6b_2t^2(1-t)^2\\
	&+4b_3t^3(1-t)+b_4t^4,
\end{align*}
where
\begin{align*}
	b_0=&-576m^7-5952m^6-23936m^5-56256m^4\\
	&-112320m^3-164736m^2-137216m-71680,\\[1mm]
	b_1=&-12m^8-348m^7-3172m^6-14092m^5\\
	&-42216m^4-99024m^3-152512m^2-137216m-71680,\\[1mm]
	b_2=&-6m^8-\frac{412}{3}m^7-1280m^6-7588m^5\\
	&-31552m^4-85728m^3-140288m^2-137216m-71680,\\[1mm]
	b_3=&-12m^7-388m^6-4424m^5-24264m^4\\
	&-72432m^3-128064m^2-137216m-71680,\\[1mm]
	b_4=&-m^8-40m^7-608m^6-4600m^5-20352m^4\\
	&-59136m^3-115840m^2-137216m-71680.
\end{align*}
Since all coefficients are negative, the trace $\tr(DS(\ell, r))$ is strictly decreasing for $1\le r\le \frac{n}{2}-1$.
\qed

\section{The second laziest tree}
In this section, we prove the second laziest tree on $n$ vertices is the double star $DS(n-3,1)$. This is verified computationally for all trees with $n\le 18$, so most of our results assume $n\ge 19$. 

We start with two lemmas on the eigenprojections of a Laplacian matrix. For any subset $S$ of vertices, let $\chi_S$ denote the characteristic vector of $S$.

\begin{lemma}\label{W_vv}
	Let $X$ be a graph with Laplacian spectral decomposition
	\[L = \sum_{\lambda} \lambda E_{\lambda}.\] 	
	Given $\omega>0$, let
	\[W= \sum_{0<\lambda\le \omega} E_{\lambda}.\]
	For any vertex $v$ with $\deg(v)\ge \omega$, we have
	\[W_{vv}\le \frac{\deg(v)}{\deg(v) + (\deg(v)-\omega)^2}.\]
\end{lemma}
\proof 
Note that
\[LW = \sum_{0<\lambda\le \omega} \lambda E_{\lambda}\preceq\omega \sum_{0<\lambda\le \omega} E_{\lambda} = \omega W.\]
Thus
\begin{align*}
	(\deg(v)-\omega ) W_{vv}&=	e_v^T (D-\omega I)We_v 
	\le  e_v^T A W e_v
	= \sum_{w\sim v} e_w^T We_v.
\end{align*}
Since $\deg(v)\ge \omega$, the left hand side is nonnegative. On the other hand, by the Cauchy-Schwarz inequality,
\begin{align*}
	\abs{\sum_{w\sim v} e_w^T W e_v}^2 &= \ip{We_v[N(v)]}{\chi_{N(v)}}^2  \\
	&\le \ip{We_v[N(v)]}{We_v[N(v)]} \ip{\chi_{N(v)}}{\chi_{N(v)}}\\
	&\le \deg(v) (\ip{We_v}{We_v} - W_{vv}^2)\\
	&= \deg(v) W_{vv}( 1- W_{vv}).
\end{align*}
Hence
\[(\deg(v)-\omega)^2 W_{vv}^2 \le \deg(v) W_{vv} (1- W_{vv}),\]
from which it follows that
\[W_{vv} \le \frac{\deg(v)}{\deg(v) + (\deg(v)-\omega)^2}.\tag*{\sqr53}\]

To prove our second result, we need the following fact from linear algebra.

\begin{lemma}\label{linalg}
	Let $U$ and $V$ be two finite-dimensional vector spaces. Then 
	\[U = (U\cap V ) \oplus \projonto{U}{V^{\perp}}.\]
\end{lemma}
\proof
By the direct sum decomposition
\[U = (U \cap V) \oplus ((U\cap V)^{\perp} \cap U),\]
and the identity
\[(U\cap V)^{\perp} = U^{\perp} + V^{\perp},\]
it suffices to show that 
\[\projonto{U}{V^{\perp}} = (U^{\perp} + V^{\perp}) \cap U.\]
First let $w\in \projonto{U}{V^{\perp}}$. Then $w\in U$, and we can write
\[w = \projonto{U}{v'}\]
for some $v'\in V^{\perp}$. The projection of $w - v'$ onto $U$ is
\[\projonto{U}{w -  v'}=w -  \projonto{U}{v'} = w - w=0.\]
Thus,
\[w = u'+v'\]
for some $u'\in U^{\perp}$. Therefore 
\[w\in (U^{\perp} + V^{\perp}) \cap U.\]
Conversely, let $w\in (U^{\perp} + V^{\perp}) \cap U$. Then $w\in U$, and we can write it as
\[w = u'+v'\]
for some $u\in U^{\perp}$ and $v'\in V^{\perp}$. The projection of $w$ onto $U$ is
\[w = \projonto{U}{w} = \projonto{U}{u'+v'}= \projonto{U}{u'} +  \projonto{U}{v'} =  \projonto{U}{v'}.\]
Therefore
\[w\in \projonto{U}{V^{\perp}}.\tag*{\sqr53}\]

We now consider the sign of some off-diagonal entries of eigenprojections of trees. 

\begin{lemma}\label{off-neg}
	Let $T$ be a tree and let $\lambda$ be a nonzero Laplacian eigenvalue.	For any $u$ such that $(E_{\lambda})_{uu}=0$ and any distinct neighbors $v,w$ of $u$, we have
	\[(E_{\lambda})_{v,w}\le 0.\]
\end{lemma}
\proof
Since $(E_{\lambda})_{uu}=0$, and the matrix $L\backslash u$ is block-diagonal with each block indexed by a component of $T-u$, we have
\[\col(E_{\lambda}) =\ker(\lambda I-L) =  \ker(\lambda I-L\backslash u) \cap \chi_{N(u)}^{\perp}.\]
Apply Lemma \ref{linalg} with $U=\ker(\lambda I-L\backslash u)$ and $V = \chi_{N(u)}^{\perp}$. Let $P$ be the projection onto $U$. 
Since
\[\projonto{U}{V^{\perp}}  =  \projonto{U}{\col(\chi_{N(u)}})  = \col(P\chi_{N(u)}),\]
we have
\[E_{\lambda} = \proj{U\cap V} = \proj{U} - \proj{\col(P\chi_{N(u)})}= P - \frac{1}{\deg(u)} P \chi_{N(u)}\chi_{N(u)}^T P.\]
Therefore
\begin{align*}
	(E_{\lambda})_{v,w}&=P_{vw} - \frac{1}{\deg(u)} e_v^T P \chi_{N(u)} \chi_{N(u)}^T P e_w\\
	&=-\frac{1}{\deg(u)} P_{vv} P_{ww}\\
	&\le 0. \tag*{\sqr53}
\end{align*}

\subsection{Nontrivial bouquets}
Given a graph $X$, let $P$ and $Q$ be the set of pendent vertices and quasi-pendent vertices in $X$, respectively. Let $R=V(X)-P-Q$. Let $p=\abs{P}$ and $q=\abs{Q}$. For each $u\in Q$, let
\[P(u) = N(u)\cap P,\quad p_u = \abs{P(u)}.\]
A \textsl{bouquet} of $X$ with support $u$ is the star induced by a quasipendent vertex $u$ and $P(u)$. We say a bouquet is \textsl{nontrivial} if $p_u\ge 2$. 

In this subsection, we show that if a tree $T$ on $n$ vertices satisfies that $\tr(T)\ge \tr(DS(n-3,1))$, then $T$ has at most one nontrivial bouquet. To start, we cite the following result due to Faria \cite{Faria1985}.

\begin{theorem}\cite{Faria1985}\label{Faria}
Let $X$ be a connected graph with Laplacian matrix $L$. Let $u\in Q$ and $v,w\in P(u)$. Then
	\[L (e_v - e_w) = e_v - e_w.\]
	Hence the multiplicity of eigenvalue $1$ is at least $p-q$.
\end{theorem}

Following \cite{Grone1990a}, we will call the vectors 
\[\{ e_v-e_w: v, w\sim u \text{ for some } u\in Q\}\]
the \text{Faria vectors}. We may write
\begin{equation}\label{E1split}
	E_1 = F + G,
\end{equation}
where $F$ is the projection onto the space spanned by the Faria vectors, and $G$ is the projection onto the space spanned by the $1$-eigenvectors that are constant on $P(u)$ for each $u\in Q$.

We prove two technical lemmas that lead to a combinatorial upper bound for the laziness of a tree.

\begin{lemma}
	Let $X$ be a graph. Let $E_1=F+G$ as defined in \eqref{E1split}. Let $C$ be the vertex set of a component of 
	\[X - \{u\notin Q: G_{uu}=0\}.\]
	Let $\Delta$ be the diagonal matrix indexed by $C\cap R$ where
	\[\Delta_{ww} = \abs{N(w) \cap C \cap Q}.\]
	Then 
	\[\tr(G[C\cap R]) \ge \tr(\Delta G[C\cap R]). \]
\end{lemma}
\proof
From 
\[LG = G,\]
we get
\[L[C\cap R] G [C\cap R] + L[C\cap R, \comp{C\cap R}] G[\comp{C\cap R}, C\cap R] = G[C\cap R].\]
Suppose for some $a\in C\cap R$ and $b\notin C\cap R$ we have
\[e_a^T L[C\cap R, \comp{C\cap R}] e_b \ne 0.\]
Then $a\sim b$, and it must be that $G_{bb}=0$. Thus $Ge_b=0$ and
\[e_b^TG[\comp{C\cap R}, C\cap R] e_a=0.\]
It follows that
\[\tr(L[C\cap R, \comp{C\cap R}] G[\comp{C\cap R}, C\cap R] )=0.\]
Consequently, 
\[\tr(L[C\cap R] G [C\cap R]) =  \tr(G[C\cap R]).\]
Note that for any $w\in C\cap R$, 
\[(N(w)\cap C\cap Q) \cup (N(w) \cap C \cap R) \sbs N(w).\]
Splitting the neighbors of $w$ into 
\[N(w) = (N(w) \cap C \cap Q) \cup (N(w) \cap C \cap R) \cup (N(w)\setminus C) \]
shows
\[L[C\cap R] \succcurlyeq \Delta + L(X[C\cap R]) \succcurlyeq  \Delta .\]
Hence
\[ L[C\cap R] \circ G [C\cap R] \succcurlyeq \Delta \circ G[C\cap R].\]
Therefore
\begin{align*}
	\tr(G[C\cap R]) &=\tr(L[C\cap R] G [C\cap R])\\
	&=\elsm(L[C\cap R] \circ G[C\cap R])\\
	&\ge \elsm(\Delta \circ G[C\cap R])\\
	&=\tr(\Delta G[C\cap R]). \tag*{\sqr53}
\end{align*}

For ease of notation, let 
\[\tr_S(M):=\tr(M[S])\]
denote the restriction of the trace of a square matrix $M$ to a subset $S$ of indices. 

\begin{lemma}
	Let $T$ be a tree. Let $E_1 = F+G$ as defined in \eqref{E1split}. We have
	\[\tr((I-E_0-E_1- F)\circ  G)\ge  0.\]
\end{lemma}
\proof
Let $C$ be the vertex set of a component of 
\[T - \{u\notin Q: G_{uu}=0\}.\]
	Let $\Delta$ be the diagonal matrix indexed by $C\cap R$ where
\[\Delta_{ww} = \abs{N(w) \cap C \cap Q}.\]
We show that 
\[\tr_C((I-E_0-E_1- F)\circ  G)\ge 0.\]
Since $G_{uu}=0$ for any $u\in Q$, it suffices to show
\[\tr_{C\cap P}((I-E_0-E_1- F)\circ  G)+\tr_{C\cap R}((I-E_0-E_1- F)\circ  G)\ge 0. \]
\begin{enumerate}[(i)]
	\item For the restriction to $C\cap R$, note that for any $w\in R$, we have
	\[F_{ww}=0.\]
	Applying Lemma \ref{W_vv} with $\omega=1$ gives
	\[(E_1)_{ww} \le \frac{2}{2+(2-1)^2}= \frac{2}{3}.\]
	Thus
	\[((I-E_0-E_1- F)\circ  G)_{ww} =\left( 1- \frac{1}{n} - (E_1)_{ww} \right) G_{ww}\ge  (\frac{1}{3}-\frac{1}{n}) G_{ww}. \]
	That is,
	\[\tr_{C\cap R}((I-E_0-E_1- F)\circ  G)\ge \left(\frac{1}{3}-\frac{1}{n}\right) \tr_{C\cap R}(G)\]
	\item For the restriction to $C\cap P$, let $u\in C\cap Q$ and $v\in P(u)$. Then 
	\[((I-E_0-E_1- F)\circ  G)_{vv} \ge -(F \circ G)_{vv} = -\left(1-\frac{1}{p_u}\right) G_{vv} \]
	Hence
	\[\tr_{P(u)} ( (I-E_0-E_1-F) \circ G) \ge  -(p_u-1) G_{vv}.\]
	Since
	\[0 = e_u^T G = p_u e_v^T G + \sum_{w\in N(u)\cap C\cap R} e_w^T G,\]
	we obtain 
	\begin{align*}
		p_u^2 G_{vv} &= \ip{G\chi_{N(u)\cap C\cap R}}{\chi_{N(u)\cap C\cap R} }
		=\elsm{(G[N(u)\cap C\cap R])}
	\end{align*}
	Applying Lemma \ref{off-neg} with $\lambda=1$ and the fact that $F[R]=0$, we see that the off-diagonal entries of $G[N(u)\cap C\cap R]$ must be nonpositive, and so 
\[\elsm{(G[N(u)\cap C\cap R])}\le \tr(G[N(u)\cap C\cap R]).\]
Hence, 
\begin{align*}
	\tr_{P(u)} ((I-E_0-E_1-F)\circ G)&\ge -\frac{p_u-1}{p_u^2} \tr(G[N(u)\cap C\cap R])\\
	&\ge -\frac{1}{4} \tr(G[N(u)\cap C\cap R]).
\end{align*}
Finally, summing this over all quasi-pendent vertices in $C$ and counting incidences in two ways yields
\begin{align*}
	\sum_{u\in C\cap Q} \tr_{P(u)} ((I-E_0-E_1-F)\circ G) & \ge -\frac{1}{4} \sum_{u\in C\cap Q} \tr(G[N(u)\cap C\cap R])\\
	&=-\frac{1}{4} \sum_{u\in C\cap Q} \sum_{w\in N(u)\cap C\cap R} G_{ww}\\
	&=-\frac{1}{4} \sum_{w\in C\cap R} \sum_{u\in N(w)\cap C\cap Q} G_{ww}\\
	&=-\frac{1}{4}\tr_{C\cap R}(\Delta G)\\
	&\ge -\frac{1}{4}\tr_{C\cap R} (G).
\end{align*}
\end{enumerate}
Therefore 
\[\tr_{C}(I-E_0-E_1- F)\circ  G)+\tr_{C\cap R}((I-E_0-E_1- F)\circ  G)\ge \left(\frac{1}{12}-\frac{1}{n}\right) \tr_{C\cap R}( G),\]
which is nonnegative for $n\ge 12$.
\qed

Thus we get an upper bound for $\tr(T)$ in terms of $p$, $q$ and $\{p_u: u\in Q\}$.

\begin{corollary}\label{cell-bound}
	For any tree $T$ we have
	\begin{align*}
		\tr(T)&\le \tr(E_0^{\circ 2}) + \tr(E_1^{\circ 2}) + \tr((I-E_0-E_1)^{\circ 2})\\
		&\le \tr(E_0^{\circ 2}) + \tr(F^{\circ 2}) + \tr((I-E_0 - F)^{\circ 2})\\
		&=n-2q -2 + 2\sum_{u\in Q} \frac{1}{p_u} + \frac{2(p-q+1)}{n}
	\end{align*}
\end{corollary}
\proof 
The first inequality follows from the difference
\begin{align*}
	&F^{\circ 2} + (I-E_0 - F)^{\circ 2} - E_1^{\circ 2} - (I-E_0-E_1)^{\circ 2} \\=&2((I-E_0-E_1-F)\circ G)\\\ge& 0,
\end{align*}
and the second inequality follows from
\[F[P(u)] = I - \frac{1}{p_u} J\]
for any quasi-pendent vertex $u$.
\qed

Comparing this to $ \tr(DS(n-3,1))$ shows that any surviving tree must have the following structure.

\begin{corollary}\label{unique_bouquet}
	Let $T$ be a tree on $n\ge 19$ vertices with diameter at least $3$, and suppose 
	\[\tr(T)\ge \tr(DS(n-3,1)).\]
	Then $T$ has at most one nontrivial bouquet. Thus either  $T= DS(n-3,1)$, or $T$ has diameter at least $4$. Moreover, the support $u$ of the largest (possibly trivial) bouquet of $T$ satisfies
	\[p_u \le 2 \text{ or } p_u\ge 8.\]
\end{corollary}
\proof
Suppose, for a contradiction, that there are at least two nontrivial bouquets. Then $p-q\ge 2$. Moreover, by Corollary \ref{cell-bound},
\begin{align*}
	\tr(T)&\le n-2q -2 + 2\sum_{u\in Q} \frac{1}{p_u} + \frac{2(p-q+1)}{n}\\
	&\le n - 2q -2+2  \left(\frac{1}{p-q} + \frac{1}{2} + q-2\right) + \frac{2(p-q+1)}{n}\\
	&=n-5 + \frac{2}{p-q}+\frac{2(p-q+1)}{n}.
\end{align*}
Thus for $n\ge 5$,
\begin{align*}
	\tr(T) - \tr(DS(n-3,1)) &< \tr(T) - \left(n-\frac{14}{5}-\frac{56}{25n}\right)\\
	&\le \frac{2(p-q+1)}{n} + \frac{56}{25n} + \frac{2}{p-q}- \frac{11}{5}.
\end{align*}
A derivative argument shows that for $n\ge 19$, the right hand side is nonpositive for $p-q\ge 2$, a contradiction. Thus $T$ has at most one nontrivial bouquet. If $T\ne DS(n-3,1)$, then it must have diameter at least $4$. Now let $u$ be the support of a largest bouquet. Corollary \ref{cell-bound} then yields 
\begin{align*}
	\tr(X) 	&\le n-4+\frac{2}{p_u} + \frac{2p_u}{n}.
\end{align*}
For $n\ge 19$, the right hand side is below $\tr(DS(n-3,1))$ if $p_u\in\{3,4,5,6,7\}$.
\qed

\subsection{Eigenvalues in an interval}
For any graph $X$ on $n$ vertices, let 
\[0=	\theta_1(X) \le \theta_2(X) \le \cdots \le \theta_n(X)\]
be the eigenvalues of $L(X)$, not necessarily distinct, in non-decreasing order. We cite a standard result on Laplacian edge interlacing (see, for example, \cite[Theorem 13.6.2]{Godsil2001}).

\begin{theorem}
	Let $X$ be a graph on $n$ vertices. Let $e$ be an edge in $X$. Then for $i=1,2,\cdots,n-1$, we have
	\[\theta_i(X)\le \theta_{i+1}(X-e).\]
\end{theorem}

Repeated application of edge interlacing yields the following.

\begin{lemma}\label{edge-interlacing}
	Let $T$ be a tree with $n$ vertices. Let $Y$ be a subtree of $T$ with $m$ vertices.  Then for $i=1,2,\cdots, m$, we have
	\[\theta_i(T) \le \theta_i(Y) .\]
\end{lemma}

We will also need the Laplacian spectrum of a path for this subsection.
\begin{theorem}\cite{Brouwer2012}\label{path}
	The $i$-th smallest Laplacian eigenvalue of $P_n$ is
	\[\theta_i(P_n) = 2\left(1-\cos \frac{(i-1)\pi}{n}\right).\]
	For $i\ge 2$, a corresponding unit eigenvector has entries
	\[\sqrt{\frac{2}{n}}\cos \frac{(2k-1)(i-1)\pi}{2n},\quad k=1,2,\cdots, n.\]
\end{theorem}

Combining this with the interlacing result yields a diameter bound.
\begin{theorem}\label{diameter_bound}
	Let $T$ be a tree with diameter $d$. For $i=1,2,\cdots,d+1$,
	\[\theta_i(T) \le \theta_i(P_{d+1})=2\left(1-\cos\frac{(i-1)\pi}{d+1}\right) .\]
	In particular, if $d\ge 4$, the algebraic connectivity of $T$ satisfies
	\[\theta_2(T)\le \frac{3-\sqrt{5}}{2}.\]
\end{theorem}

We now apply Lemma \ref{W_vv} with $\omega = \frac{3-\sqrt{5}}{2}$ to show 
\[W=\sum_{0< \lambda \le \omega} E_{\lambda}\]
has rank $1$ for any surviving tree in competition with $DS(n-3,1)$.  Our first result is an upper bound for $\tr(W^{\circ 2})$ in terms of $\tr(W)$. 

\begin{lemma}
	Let $X$ be a connected graph with Laplacian spectral decomposition
	\[L = \sum_{\lambda} \lambda E_{\lambda}.\]	
	Let $\omega = \frac{3-\sqrt{5}}{2}$, and
	\[W = \sum_{0<\lambda\le \omega} E_{\lambda}.\]
Then
	\[\tr(W^{\circ 2}) \le \frac{3}{5}\tr(W).\]
\end{lemma}
\proof 
The inequality holds trivially if $W=0$. So assume $W\ne 0$. By Lemma \ref{W_vv}, any vertex $v\notin P \cup Q$ satisfies
\[W_{vv}\le \frac{2}{2+(2-\omega)^2}< \frac{3}{5},\]
and so
\[\tr(W[\comp{P\cup Q}]^{\circ 2}) \le \frac{3}{5} \tr(W[\comp{P\cup Q}]) .\]
Thus it suffices to show that for each $u\in Q$,
\[\tr(W[\{u\}\cup P(u)]^{\circ 2}) \le \frac{3}{5} \tr(W[\{u\}\cup P(u)]),\]
that is, for each $v\in P(u)$,
\begin{equation}\label{tr_bouquet}
	W_{uu}^2 + p_u W_{vv}^2 \le \frac{3}{5} (W_{uu} + p_u W_{vv}).
\end{equation}
The eigen-equation at vertex $v$ gives
\[(E_{\lambda})_{uu}=(1-\lambda)^2 (E_{\lambda})_{vv}.\]
Since $(1-\omega)^2=\omega$, we have
\[\omega (E_{\lambda})_{vv} \le (1-\lambda)^2 (E_{\lambda})_{vv} = (E_{\lambda})_{uu} \le (E_{\lambda})_{vv}.\]
Note that $\theta_n(W)=1$ as $W$ is the orthogonal projection onto a nonzero subspace. By the min-max theorem,
\begin{align*}
	p_u&= p_u \theta_n(W)
	\ge \chi_{P(u)}^T W \chi_{P(u)}  
	=\elsm(W[P(u)]) 
	=p_u^2 W_{vv}.
\end{align*}
Hence
\begin{equation}\label{1/pu}
	W_{uu}\le W_{vv} \le \frac{1}{p_u}.
\end{equation}
For $p_u\ge 2$, this immediately implies
\[W_{uu}^2 <\frac{3}{5} W_{uu},\quad W_{vv}^2 < \frac{3}{5} W_{vv},\]
and so \eqref{tr_bouquet} holds. Now suppose $p_u=1$. If $W_{vv}=0$, then $W_{uu}=0$ and \eqref{tr_bouquet} is vacuously true. So assume $W_{vv}\ne 0$, and let
\begin{equation}\label{t}
	 t = \frac{W_{uu}}{W_{vv}},\quad \omega\le t \le 1.
	 \end{equation}
Then \eqref{tr_bouquet} is equivalent to 
\begin{equation*}
	W_{vv} \le \frac{3}{5}\frac{t+1}{t^2+1}.
\end{equation*}
Let $\alpha =1-\omega$.
We claim that
\[W_{vv}\le \frac{(1+\alpha)^2}{(t+\alpha)^2+(1+\alpha)^2}\le  \frac{3}{5}\frac{t+1}{t^2+1}.\] 
The second inequality follows from direct comparison and that $\alpha^2=1-\alpha$.
For the first inequality, note that 
\[0\preceq LW\preceq \omega W, \]
from which we obtain
\[LW(\omega I - L)W\succeq 0.\]
Hence
\begin{align*}
	0&\le \ip{WLe_v}{(\omega I-L) We_v}\\
	&=\ip{W(e_v-e_u)}{W(e_u - \alpha e_v)}\\
	&=(1+\alpha)W_{uv}-W_{uu}-\alpha W_{vv}.
\end{align*}
Isolating $W_{uv}$ gives
\begin{equation}\label{Wuv1}
	W_{uv} \ge \frac{W_{uu}+\alpha W_{vv}}{1+\alpha} =\frac{t+\alpha}{1+\alpha} W_{vv}.
\end{equation}
On the other hand, since
\[W\preceq I,\]
we have
\begin{equation}\label{Wuv2}
	0\le \det((I-W)[u,v]) = (1-W_{uu})(1-W_{vv})-W_{uv}^2.
	\end{equation}
Combining \eqref{t}, \eqref{Wuv1} and \eqref{Wuv2} gives the desired upper bound for $W_{vv}$.
\qed

It follows that for any connected graph $X$ on $n\ge 19$ vertices such that $\tr(X)\ge \tr(DS(n-3,1))$, there is at most one eigenvalue in the interval $(0, \frac{3-\sqrt{5}}{2}]$, that is, the algebraic connectivity, with multiplicity one.

\begin{corollary}\label{tr1}
	Let $X$ be a connected graph on $n\ge 19$ vertices with 
	\[\tr(X)\ge \tr(DS(n-3,1)).\]
	 Let $\omega=\frac{3-\sqrt{5}}{2}$. Then
	 $X$ has at most one eigenvalue in the interval $(0, \omega]$. Moreover, if $X$ is a tree $T$ with diameter at least $4$, then 	 
	\[\theta_2(T) \le  \omega <\theta_3(T).\]
\end{corollary}
\proof
Let
\[W = \sum_{0<\lambda \le \omega} E_{\lambda}.\]
We have
\begin{align}
	\tr(X)&\le \tr(E_0^{\circ 2}) + \tr(W^{\circ 2}) + \tr((I-E_0-W)^{\circ 2}) \notag \\
	&= \tr(E_0^{\circ 2}) + \tr((I-E_0)^{\circ 2})  - 2\tr((I-E_0) \circ W) + 2\tr(W^{\circ 2}) \notag \\
	&=n-2+\frac{2}{n}  - \frac{2(n-1)}{n} \tr(W) + 2\tr(W^{\circ 2}) \label{simple}\\
	&\le n-2+\frac{2}{n} - \frac{2(n-1)}{n} \tr(W) + \frac{6}{5}\tr(W) \notag \\
	&=n-2+\frac{2}{n} -\left(\frac{4}{5}- \frac{2}{n} \right)\tr(W). \notag
\end{align}
Thus by Theorem \ref{dsn-31}, for
\[\tr(X) \ge  \tr(DS(n-3,1))>n-\frac{14}{5} - \frac{56}{25n},\]
we must have $\tr(W)\le 1$. If, in addition, $X$ is a tree with diameter at least $4$, then its algebraic connectivity lies below $\omega$, from which the claimed inequality follows.
\qed

A \textsl{caterpillar} is a tree whose removal of leaves gives a path. Corollary \ref{tr1}  shows that any tree on $n$ vertices that beats or tie $DS(n-3,1)$ must be a caterpillar. In fact, we can say something stronger.

\begin{corollary}
	Let $T$ be a tree on $n\ge 19$ vertices with diameter at least $4$, and suppose 
	\[\tr(T)\ge \tr(DS(n-3,1)).\]
	Then $T$ is a caterpillar with diameter at most $8$, and it has a unique nontrivial bouquet that contains at least  $8$ leaves.
\end{corollary}
\proof
Let $d$ be the diameter of $T$. Let
\[\omega = \frac{3-\sqrt{5}}{2}.\]
By Corollary \ref{tr1} and Theorem \ref{diameter_bound}, 
\[\omega< \theta_3(T)\le \theta_3(P_{d+1})=2\left(1-\cos\frac{2\pi}{d+1}\right),\]
so we must have $d\le 8$. If $T$ is not a caterpillar, then $T$ contains the spider graph $S(2,2,2)$ obtained from $S_4$ by subdividing each edge once. Thus by Lemma \ref{edge-interlacing},
\[\theta_3(T)\le \theta_3(S(2,2,2)) = \omega,\]
a contradiction. Therefore, $T$ is a caterpillar. Let $P_{d+1}$ be a longest path in $T$. Since $n\ge 19$ and $d\le 8$, at least $10$ vertices are not on $P_{d+1}$, and they are adjacent to at most $7$ vertices on $P_{d+1}$. Thus for $T$ to have at most one nontrivial bouquet, some vertex $u$ on $P_{d+1}$ must have at least $4$ leaf neighbors. By Corollary \ref{unique_bouquet}, $p_u\ge 8$.
\qed

\subsection{Algebraic connectivity}
In this subsection, we complete the proof that $DS(n-3,1)$ is the second laziest tree on $n$ vertices, using the eigenprojection associated with the algebraic connectivity.

\begin{lemma}\label{3/7}
	Let $T$ be a tree on $n\ge 19$ vertices with diameter at least $4$, and suppose 
	\[\tr(T)\ge \tr(DS(n-3,1)).\]
	Let $H$ be the $\theta_2(T)$-eigenprojection of $T$. Then 
		\[\tr(H^{\circ 2})> \frac{3}{7}.\]
\end{lemma}
\proof
By Corollary \ref{tr1} and equation \eqref{simple}, $\tr(H)=1$, and
\[\tr(T) \le n-4+\frac{4}{n}+2\tr(H^{\circ 2}).\]
If $\tr(H^{\circ 2})\le \frac{3}{7}$, then for $n\ge 19$, the above inequality and Theorem \ref{dsn-31} give
\begin{align*}
	\tr(T) &\le  n - \frac{22}{7}+\frac{4}{n}
	<n-\frac{14}{5} - \frac{56}{25n}
	<\tr(DS(n-3,1)),
\end{align*}
which contradicts our choice of $T$. Hence $\tr(H^{\circ 2})> \frac{3}{7}$.
\qed

A \textsl{Fiedler vector} for a connected graph $X$ is an eigenvector corresponding to its algebraic connectivity $\theta_2(X)$. A tree is \textsl{type I} if some Fiedler vector vanishes at a vertex, and \textsl{type II} otherwise. A \textsl{branch} of a tree $T$ at a vertex $v$ is a component of $T-v$. For a type I tree, a branch at $a$ is \textsl{active} if some Fiedler vector is nonzero on it, and \textsl{passive} otherwise. The following theorem collects the results from \cite{Fiedler1975b,Merris1987,Grone1987,Kirkland1996}.

\begin{theorem}\cite{Fiedler1975b,Merris1987,Grone1987,Kirkland1996}\label{typeItypeII}
	Let $T$ be a tree.
	\begin{enumerate}[(i)]
		\item If $T$ is \textsl{type I}, then there is a unique vertex $a$ such that, for every Fiedler vector $z$, we have $z_a=0$ and $a$ is adjacent to $\supp(z)$. Moreover, the multiplicity of $\theta_2(T)$ is one less the number of active branches at $a$. For any Fiedler vector $z$ of $T$ and any active branch $B$ at $a$, the entries of $z[B]$ have the same sign, and
		\[\theta_2(T) z[B]=L[B]z[B] = \theta_1(L[B])z[B].\]
		\item If $T$ is \textsl{type II}, then there is a unique edge $ab$ such that $z_a z_b<0$ for every Fiedler vector $z$. Moreover, $\theta_2(T)$ is simple. If $B_+$ is the branch at $b$ containing $a$, and $B_-$ is the branch at $a$ containing $b$, then the entries of $z[B_+]$ have the same sign, and the entries of $z[B_-]$ have the same sign. Moreover, there is $\eta>0$ such that 
		\[\theta_2(T) z[B_+]=(L[B_+]+ \eta E_{aa})z[B_+] = \theta_1(L[B_+]+ \eta E_{aa}) z[B_+]\]
		and
		 \[  \theta_2(T) z[B_-]=\left(L[B_-] + \frac{1}{\eta} E_{bb}\right)z[B_-] = \theta_1\left(L[B_-] + \frac{1}{\eta} E_{bb}\right)z[B_-].\]
	\end{enumerate}
\end{theorem}

The vertex $a$ in  (i) and the edge $ab$ in (ii) is called the \textsl{characteristic vertex} and the  \textsl{characteristic edge}, respectively; we will refer to either as the \textsl{spectral center}.

Let $x$ and $y$ be two vectors in $\re^n$. Let $x^{\downarrow}$  and $y^{\downarrow}$ denote the vectors where $x^{\downarrow}_i$ is the $i$-th largest entry of $x$ and $y^{\downarrow}_i$ is the $i$-th largest entry of $y$. We say $y$ \textsl{majorizes} $x$, denoted $x\prec y$, if 
and for $k=1,2,\cdots n-1$,
\begin{equation}\label{maj}
	\sum_{i=1}^k x^{\downarrow}_i \le \sum_{i=1}^{k} y^{\downarrow}_i.
\end{equation}
and
\[\sum_{i=1}^n x^{\downarrow}_i = \sum_{i=1}^{n} y^{\downarrow}_i.\]
The majorization is \textsl{tight} if equality holds in \eqref{maj} for each $k=1,2,\cdots, n-1$.
For more background on majorization, see \cite[Section 4.3]{Horn2012}.

The following theorem combines two results due to Marshall, Olkin and Proschan \cite[Theorems 2.4 and 2.5]{Marshall1967}.

\begin{theorem}\cite{Marshall1967}\label{dec_ratio}
	Let $x$ and $y$ be two vectors in $\re^n$ with positive entries. Suppose
	\[\frac{y^{\downarrow}_1}{x^{\downarrow}_1} \ge \frac{y^{\downarrow}_2}{x^{\downarrow}_2}\ge \cdots\ge  \frac{y^{\downarrow}_{n}}{x^{\downarrow}_{n}}.\]
	Then 
	\[\frac{x}{\ip{x}{\one}} \prec \frac{y}{\ip{y}{\one}}.\]
	Moreover, 
	\[\sqrt[r]{\frac{\ip{x^{\circ r}}{\one}}{\ip{y^{\circ r}}{\one}}}\]
	is non-increasing in $r$, and it is strictly decreasing if and only if $x^{\downarrow}$ and $y^{\downarrow}$ are not parallel.
\end{theorem}

Given a tree $T$ rooted at a vertex $r$, the \textsl{height} of $T$ is the maximum distance from a vertex in $T$ to $r$. For any $\tau>0$, define
\[\theta_1(T, r, \tau ) = \theta_1(L(T) + \tau E_{rr}).\]
Note that $L(T)+\tau E_{rr}$ is positive definite, and its smallest eigenvalue is simple with a positive eigenvector. We will need the following formula, due to Kirkland, Neumann and
Shader \cite{Kirkland1996}, for the inverse of $L(T) + \tau E_{rr}$.

\begin{theorem}\cite{Kirkland1996}\label{path_formula}
Let $T$ be a tree rooted at $r$. For any two vertices $u$ and $v$ in $T$, let $P_{ru}$ and $P_{rv}$ denote the paths from $r$ to $u$ and from $r$ to $v$, respectively. Then
\begin{equation*}
	\left(L(T)+\tau E_{rr}\right)^{-1}_{uv}
	=\frac{1}{\tau}+\abs{E(P_{ru})\cap E(P_{rv})}.
\end{equation*}
\end{theorem}

To rule out the remaining candidates for the second laziest tree, we prove a majorization result on the eigenvectors of $\theta_1(T,r,\tau)$.

\begin{lemma}\label{majorization}
 Let $T$ be a tree on $n$ vertices rooted at vertex $r$ with height $h\ge 1$. Let $P_h$ and $P_{h+1}$ be two paths rooted at their endpoints $r_0$ and $r_1$, respectively. Given $\tau>0$, let $x, y, z\in \re^n$ be, respectively, the unit positive eigenvector of $\theta_1(T, r, \tau)$, the unit positive eigenvector of $\theta_1(P_{h+1}, r_1, \tau)$ padded with zeros, and the unit positive eigenvector of $\theta_1(P_h, r_0, 1)$ padded with zeros. Then the following hold.
 \begin{enumerate}[(i)]
 	\item $y\circ y$ majorizes $x\circ x$, and 
 	\[\ip{x\circ x}{x\circ x}\le \ip{y\circ y}{y\circ y}.\]
 	Moreover, the majorization and inequality are tight if and only if the rooted tree $(T, r)$ is isomorphic to the rooted path $(P_{h+1}, r_1)$.  
 	\item $z\circ z$ strictly majorizes $y\circ y$, and 
 	\[\ip{y\circ y}{y\circ y}<\ip{z\circ z}{z\circ z}.\]
 \end{enumerate}
\end{lemma}
\proof
Let $v$ be a vertex that attains the smallest entry in $x$. If $v\ne r$, then all neighbors of $r$ lie in the same branch at $r$, and so
\[\theta_1(T, r, \tau)x_v = e_v^TL(T)x = \sum_{u\sim v}(x_v-x_u) \le 0< \theta_1(T, r, \tau),\]
a contradiction. Hence $v=r$, that is, $x^{\downarrow}_n = x_r$. For any $k<n$, let $S$ be the set of vertices in $T$ that attain the largest $k$ entries in $x$. Then $r\notin S$, and so
\begin{align}
	\theta_1(T, r, \tau)\sum_{i=1}^k x^{\downarrow}_i&=\ip{\chi_S}{ \theta_1(T,r, \tau) x} \notag \\
	 &= \ip{\chi_S}{ L(T) x}+ \ip{\chi_S}{\tau E_{rr} x} \notag \\
	&=\ip{\chi_S}{ L(T) x} \notag \\
	&=\sum_{\substack{u\in S,\;w\notin S\\u\sim w}}(x_u-x_w)\notag \\
	&\ge x^{\downarrow}_k - x^{\downarrow}_{k+1}, \label{boundary}
\end{align}
Note that equality in \eqref{boundary} holds simultaneously for each $k=1,2,\cdots,h$ if and only if $T$ is a path $P_{h+1}$ rooted at an endpoint. Thus we also have
\[\theta_1(P_{h+1}, r_1,  \tau ) \sum_{i=1}^k y^{\downarrow}_i   = y^{\downarrow}_k - y^{\downarrow}_{k+1}.\]
Identify $P_{h+1}$ with a longest path in $T$ rooted at $r$. By Theorem \ref{path_formula}, 
\[(L(T)+\tau E_{rr})^{-1}[P_{h+1}]=(L(P_{h+1}) + \tau  E_{r_1r_1})^{-1}.\]
Therefore 
\[\frac{1}{\theta_1(T,r, \tau )} \ge  \frac{1}{\theta_1(P_{h+1},r_1, \tau )},\]
that is,
\[\theta_1(T,r, \tau ) \le  \theta_1(P_{h+1},r_1, \tau ).\]
An induction shows that 
\[\frac{y^{\downarrow}_1}{x^{\downarrow}_1} \ge \frac{y^{\downarrow}_2}{x^{\downarrow}_2}\ge \cdots\ge  \frac{y^{\downarrow}_{h+1}}{x^{\downarrow}_{h+1}}.\]
Since $y^{\downarrow}_i=0$ for $i>h+1$, we have
 \[\frac{(y^{\downarrow}_1)^2}{(x^{\downarrow}_1)^2} \ge  \frac{(y^{\downarrow}_2)^2}{(x^{\downarrow}_2)^2}\ge \cdots\ge  \frac{(y^{\downarrow}_n)^2}{(x^{\downarrow}_n)^2}.\]
 By Theorem \ref{dec_ratio} (i) and 
 \[\ip{x\circ x}{\one} = \ip{y\circ y}{\one}=1,\]
 we see that
 \[x\circ x \prec y\circ y,\]
 and
 \[\sqrt{	\frac{\ip{x\circ x}{x\circ x}}{\ip{y\circ y}{y\circ y}}} =\sqrt{\frac{\ip{x^{\circ 4}}{\one}}{\ip{y^{\circ 4}}{\one}}}\le \frac{\ip{x^{\circ 2}}{\one}}{\ip{y^{\circ 2}}{\one}}=1.\]

For the second majorization, note that $L(P_h)+E_{r_0r_0}$ can be obtained from $L(P_{h+1}) + \tau E_{r_1r_1}$ by deleting the $r_1$-th row and the $r_1$-th column. Hence
\[\theta_1(P_{h+1}, r_1,  \tau) < \theta_1(P_h, r_0, 1).\]
A similar induction using
\[\theta_1(P_{h+1},r_1,  \tau ) \sum_{i=1}^k y^{\downarrow}_i   = y^{\downarrow}_k - y^{\downarrow}_{k+1},\quad \theta_1(P_h,r_0,  1 ) \sum_{i=1}^k z^{\downarrow}_i   = z^{\downarrow}_k - z^{\downarrow}_{k+1},\]
shows that 
\[\frac{z^{\downarrow}_1}{y^{\downarrow}_1} > \frac{z^{\downarrow}_2}{y^{\downarrow}_2}> \cdots>  \frac{z^{\downarrow}_{h+1}}{y^{\downarrow}_{h+1}}.\]
Squaring each term and applying Theorem \ref{dec_ratio} again proves the second part of the statement.
\qed

One consequence of the above result is that, for any tree $T$ and eigenprojection $H$ of its algebraic connectivity, we have an upper bound for $\tr(H^{\circ 2}[B])$ in terms of $\tr(H[B])$, where $B$ is any component of the graph obtained from $T$ by removing the spectral center.

\begin{theorem}\label{extremal}
Let $T$ be a tree, and let $H$ be the $\theta_2(T)$-eigenprojection.
	\begin{enumerate}[(i)]
		\item If $T$ be a type I with characteristic vertex $a$, and $B$ is a branch at $a$ with height $h$, then
		\[\tr(H[B]^{\circ 2}) \le \frac{3}{2h+3} (\tr(H[B]))^2,\]
		with equality if and only if $B$ is a $P_{h+1}$ rooted at an endpoint.
		\item If $T$ be a type II tree with characteristic edge $ab$, and $B$ be the branch at $b$ rooted at $a$ with height $h$, then 
		\[\tr(H[B]^{\circ 2}) < \frac{3}{2h+1} (\tr(H[B]))^2.\]
	\end{enumerate}
\end{theorem}
\proof
Suppose first that $T$ is type I, with characteristic vertex $a$. As the inequality trivially holds when $H[B]=0$, we may assume $B$ is an active branch. Let $r$ be the neighbor of $a$ in $B$. By Theorem \ref{typeItypeII} (i), 
\[\theta_2(T) = \theta_1(L[B])=\theta_1(B, r, 1 ),\]
and so we may write
\[H[B] = \tr(H[B]) xx^T\]
for some unit eigenvector $x$ of $\theta_1(L[B])$. Thus, if $r_1$ is an endpoint of $P_{h+1}$, and $y$ is the positive unit eigenvector of $\theta_1(P_{h+1},r_1, 1)$, then the first part of Lemma \ref{majorization} implies 
\begin{align*}
	\tr(H[B]^{\circ 2})& = \ip{x\circ x} {x\circ x} (\tr(H[B]))^2 
	\le \ip{y\circ y} {y\circ y}  (\tr(H[B]))^2.
\end{align*}
To calculate $\ip{y\circ y}{y\circ y}$, note that $P_{2h+3}$ has 
\[\pmat{y' \\ 0 \\ -y}\]
as a Fiedler vector, where $y'$ is the reverse of $y$. By Theorem \ref{path},
\[y_k = -x_{h+2+k}=\frac{2}{\sqrt{2h+3}} \sin \left(\frac{k\pi}{2h+3}\right).\]
Hence
\[\ip{y\circ y}{y\circ y} = \frac{16}{(2h+3)^2}\sum_{k=1}^{h+1} \sin^4 \left(\frac{k\pi}{2h+3} \right)= \frac{16}{(2h+3)^2} \frac{3(2h+3)}{16} = \frac{3}{2h+3}.\]
It follows that 
\[\tr(H[B]^{\circ 2}) \le \frac{3}{2h+3} (\tr(H[B]))^2,\]
with equality if and only if $B$ is a $P_{h+1}$ rooted at $a$.

Now suppose $T$ is type II, with characteristic edge $ab$. By \ref{typeItypeII} (ii), 
\[\theta_2(T) = \theta_1(L[B]+\eta E_{aa}) = \theta_1(B, r, 1+\eta)\]
for some $\eta>0$. Let $r_0$ be an endpoint of $P_h$, and $z$ be the positive unit eigenvector of $\theta_1(P_h, r_0, 1)$. Then the second part of Lemma \ref{majorization} gives
\[	\tr(H[B]^{\circ 2}) = \ip{x\circ x} {x\circ x} (\tr(H[B]))^2 
< \ip{z\circ z} {z\circ z}  (\tr(H[B]))^2.\]
By a similar argument, the entries of $z$ are
\[z_k = \frac{2}{\sqrt{2h+1}} \sin \left(\frac{k\pi}{2h+1}\right),\]
and we have
\[\ip{z\circ z}{z\circ z} = \frac{3}{2h+1}. \tag*{\sqr53}\]

Recall that if a tree $T$ on $n\ge 19$ satisfies $\tr(T)\ge \tr(DS(n-3,1))$, then its algebraic connectivity is simple, and so removing the spectral center leaves two components that are active branches. Theorem \ref{extremal} together with Lemma \ref{3/7} showsthis that at least one of these components is short.

\begin{lemma}\label{short}
Let $T$ be a tree on $n\ge 19$ vertices with diameter at least $4$ and simple algebraic connectivity. If
\[\tr(T)\ge \tr(DS(n-3,1)),\]
then one of the following holds.
\begin{enumerate}[(i)]
\item 	Suppose $T$ is type I with characteristic vertex $a$. Let $B_+$ and $B_-$ be the active branches at $a$, with heights $h_+$ and $h_-$, respectively. Then
	\[\min\{h_+, h_-\} = 1.\]
\item 	Suppose $T$ is type II with characteristic edge $ab$. Let $B_+$ and $B_-$ be the components of $T-ab$, with heights $h_+$ and $h_-$, respectively. Then
	\[\min\{h_+, h_-\}\le 2.\]
\end{enumerate}
\end{lemma}
\proof
First note that an active branch at a characteristic vertex must contain at least two vertices. Let $H$ be the $\theta_2(T)$-eigenprojection. If either condition in (i) and (ii) fails, then by Theorem \ref{extremal}, 
\[\tr(H[B_+]^{\circ 2})\le \frac{3}{7} \tr(H[B_+])^2,\quad \tr(H[B_-]^{\circ 2})\le \frac{3}{7} \tr(H[B_-])^2.\]
Since $\tr(H)=1$, we have
\begin{align*}
	\tr(H^{\circ 2})& \le  \frac{3}{7}( \tr(H[B_+])^2+ \tr(H[B_-])^2) \\
	&\le \frac{3}{7}( \tr(H[B_+])+ \tr(H[B_-]))\\
	&\le  \frac{3}{7} \tr(H) \\
	&=\frac{3}{7}.
\end{align*}
It follows from Lemma \ref{3/7} that $\tr(T)<\tr(DS(n-3,1))$.
\qed

We now rule out all remaining type I candidates.

\begin{theorem}
	 Let $\omega=\frac{3-\sqrt{5}}{2}$. Let $T$ be a type I caterpillar on $n\ge 19$ vertices with diameter at least $4$ and 
	\[\theta_2(T)\le \omega <\theta_3(T).\]
	 Suppose $T$ has exactly one nontrivial bouquet supported at $u$ with $p_u\ge 8$. Then 
	\[\tr(T)<\tr(DS(n-3,1)).\]
\end{theorem}
\proof
Assume, for a contradiction, that
\[\tr(T)\ge \tr(DS(n-3,1)).\]
Let $H$ be the $\theta_2$-eigenprojection. Let $B_+$ and $B_-$ be the active branches at the characteristic vertex, and suppose without loss of generality that $B_+$ has height $1$. If $B_+$ is a rooted $P_2$, then 
\[\omega = \theta_1(L[B_+]) = \theta_2(T) =\theta_1(L[B_-]),\]
which forces $B_-$ to be a $P_2$ as well. In this case,
\[\tr(H^{\circ 2}) \le \frac{3}{7},\]
contradicting Lemma \ref{3/7}. 
Hence $B_+$ is the nontrivial bouquet $S_{p_u+1}$ rooted at $u$. It follows that $B_-$ is  a rooted path of length at least $2$, and so
\[\tr(H[B_-]^{\circ 2})\le \frac{3}{7}.\]
For $p_u\ge 8$, reusing \eqref{1/pu} with $W=H$ gives
\[\tr(H[B_+]^{\circ 2})\le H_{vv} \tr(H[B_+])< \frac{1}{p_u} < \frac{3}{7},\]
which again contradicts Lemma \ref{3/7}.
\qed

Finally, we rule out all remaining type II candidates. 

\begin{lemma}
 Let $\omega=\frac{3-\sqrt{5}}{2}$.	Let $T$ be a type II caterpillar on $n\ge 19$ vertices with diameter at least $4$, and
		\[\theta_2(T)\le \omega <\theta_3(T).\]
		Suppose $T$ has exactly one nontrivial bouquet supported at $u$ with $p_u\ge 8$. 
		Then 
		\[\tr(T)<\tr(DS(n-3,1)).\]
\end{lemma}
\proof
Assume, for a contradiction, that
\[\tr(T)\ge \tr(DS(n-3,1)).\]
Let $H$ be the $\theta_2$-eigenprojection. Let $B_+$ and $B_-$ be the components of $T-ab$, with heights $h_+$ and $h_-$, respectively. Suppose without loss of generality that $B_+$ contains the nontrivial bouquet. Then $B_-$ is a rooted caterpillar with only trivial bouquets. 

We first narrow down the possible structures of $B_-$. 
If $B_-$ is $P_1$ or $P_2$, then
\[\theta_2(T)> \theta_1(L[B_-]) =2\left(1-\cos\frac{\pi}{2h_-+3}\right)\ge\omega,\]
contradicting our assumption. Hence $h_-\ge 2$. There are two possibilities.  If $h_-\ge 3$, then by Lemma \ref{short}, $h_+\le 2$. If $h_-=2$, then $B_-$ is either a $P_3$ rooted at an endpoint, or a $P_4$ rooted at an internal vertex, both of which satisfy
\begin{equation}\label{shortB-}
	\theta_1(L[B_-])\ge \frac{1}{7},\quad \tr(H[B_-]^{\circ 2})\le \frac{3}{7}.
	\end{equation}

We now narrow down the possible structures of $B_+$. If $h_+\le 2$, then $B_+$ is either (a) a $S_{p_u+1}$ rooted at the center, or (b) a $S_{p_u+2}$ rooted at a leaf, or (c) a $DS(p_u, 1)$ rooted at the degree-$2$ vertex, or (d) a $DS(p_u, 1)$ rooted at the branching vertex  (see Figure \ref{4cases}). 
\begin{figure}[htbp]
		\centering
		
\tikzset{
	casevertex/.style={
		circle,
		draw=black,
		fill=white,
		minimum size=14pt,
		inner sep=0pt,
		font=\small,
		line width=0.6pt
	},
	caseroot/.style={
		casevertex,
		path picture={
			\draw[black,line width=0.45pt]
			(path picture bounding box.center)
			circle[radius=5pt];
		}
	},
	casesupport/.style={casevertex},
	caserootsupport/.style={caseroot},
	caseedge/.style={
		draw=black,
		line width=0.6pt
	}
}
		\begin{subfigure}[t]{0.47\textwidth}
			\centering
			\begin{tikzpicture}[x=1.25cm,y=0.95cm]

				\path (0,-1.35)--(0,1.45);

				\node[caserootsupport] (u) at (0,0) {$u$};

				\node[casevertex] (l1) at (1.10, 1.15) {};
				\node[casevertex] (l2) at (1.40, 0) {};
				\node at (1.25,-0.58) {$\vdots$};
				\node[casevertex] (l3) at (1.10,-1.15) {};
				
				\draw[caseedge]
				(-1.15,0)--(u)
				(u)--(l1)
				(u)--(l2)
				(u)--(l3);
			\end{tikzpicture}
			\caption{\(S_{p_u+1}\) rooted at the center}
			\label{fig:star-center}
		\end{subfigure}
		\hfill
		%
		\begin{subfigure}[t]{0.47\textwidth}
			\centering
			\begin{tikzpicture}[x=1.25cm,y=0.95cm]
				\path (0,-1.35)--(0,1.45);
				
				\node[caseroot] (r) at (0,0) {$r$};
				\node[casesupport] (u) at (1.60,0) {$u$};

				\node[casevertex] (l1) at (2.50, 1.15) {};
				\node[casevertex] (l2) at (2.80, 0) {};
				\node at (2.65,-0.58) {$\vdots$};
				\node[casevertex] (l3) at (2.50,-1.15) {};
				
				\draw[caseedge]
				(-1.15,0)--(r)
				(r)--(u)
				(u)--(l1)
				(u)--(l2)
				(u)--(l3);
			\end{tikzpicture}
			\caption{\(S_{p_u+2}\) rooted at a leaf}
			\label{fig:star-leaf}
		\end{subfigure}
		
		\medskip
		
		\begin{subfigure}[t]{0.47\textwidth}
						\centering
			\begin{tikzpicture}[x=1.25cm,y=0.95cm]
				\path (0,-1.35)--(0,1.45);
				
				\node[caseroot] (r) at (0,0) {$r$};
				\node[casevertex] (s) at (0,1.20) {$s$};
				\node[casesupport] (u) at (1.60,0) {$u$};

				\node[casevertex] (l1) at (2.50, 1.15) {};
				\node[casevertex] (l2) at (2.80, 0) {};
				\node at (2.65,-0.58) {$\vdots$};
				\node[casevertex] (l3) at (2.50,-1.15) {};
				
				\draw[caseedge]
				(-1.15,0)--(r)
				(s)--(r)--(u)
				(u)--(l1)
				(u)--(l2)
				(u)--(l3);
			\end{tikzpicture}
			\caption{\(DS(p_u,1)\) rooted at the degree-\(2\) vertex}
			\label{fig:double-star-degree-two}
		\end{subfigure}
		\hfill
		%
		\begin{subfigure}[t]{0.47\textwidth}
		\centering
\begin{tikzpicture}[x=1.25cm,y=0.95cm]
	\path (0,-1.35)--(0,1.45);

	\node[caserootsupport] (u) at (0,0) {$u$};

	\node[casevertex] (v) at (1.30,0) {$w$};
	\node[casevertex] (w) at (2.60,0) {$z$};

	\node[casevertex] (l1) at (0.30,1.15) {};
	\node[casevertex] (l2) at (1.00,1.15) {};
	\node at (1.65,1.15) {$\cdots$};
	\node[casevertex] (l3) at (2.30,1.15) {};
	
	\draw[caseedge]
	(-1.15,0)--(u)
	(u)--(v)--(w)
	(u)--(l1)
	(u)--(l2)
	(u)--(l3);
\end{tikzpicture}
\caption{\(DS(p_u,1)\) rooted at the branching vertex}
\label{fig:double-star-branching}
		\end{subfigure}
		
		\caption{The four possibilities for \(B_+\) when \(h_+\le2\).
			The double circle marks the root, and the short segment to its left
			is the edge joining $B_+$ to $B_-$.}
		\label{4cases}
		
	\end{figure}
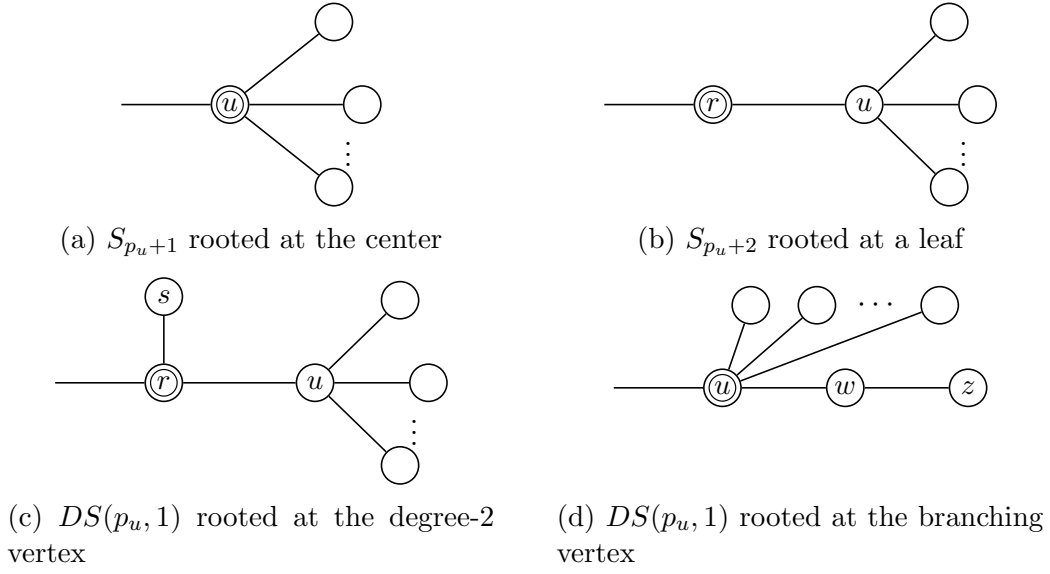
Reuse \eqref{1/pu} with $W=H$. Consider the first three cases. Let $v\in P(u)$. We have
\[H_{uu} \le H_{vv}\le \frac{1}{p_u}\le \frac{1}{8}.\]
Hence
\[\tr(H[\{u\}\cup P(u)]^{\circ 2}) < \frac{3}{7} \tr(H[\{u\}\cup P(u)]) <\frac{3}{7}.\]
This shows $\tr(H[B_+]^{\circ 2})<3/7$ in Case (a). For Cases (b) and (c), note that the eigen-equation at $u$ and the eigen-equation at $s$, if it exists, imply
\[H_{rr}<H_{vv},\quad H_{ss}<H_{vv}.\]
so we also have $\tr(H[B_+]^{\circ 2})<3/7$. For Case (d), an exact calculation using the equitable partition gives 
\[H_{ww}<\frac{3}{7},\quad H_{zz}<\frac{3}{7},\]
which again yields $\tr(H[B_+]^{\circ 2})<3/7$. Thus in all cases, $h_-=2$, and from the second part of \eqref{shortB-} we have
\[\tr(H^{\circ 2}) =\tr(H[B_+]^{\circ 2}) + \tr(H[B_-]^{\circ 2}) \le \frac{3}{7},\]
contradicting Lemma \ref{3/7}.

The only remaining possibility is $h_+\ge 3$ and $h_-=2$. If $u\notin\{a,b\}$, let $C$ be the component of $T-a$ containing $u$. Using the min-max theorem,  the fact that $p_u\ge 8$, and the first part of \eqref{shortB-}, we have 
\[\theta_1(L[C]) \le \frac{\ip{L[C]\one}{\one}}{\ip{\one}{\one}} \le \frac{1}{p_u+1}<\frac{1}{7} \le \theta_1(L[B_-]).\]
Thus
\[\theta_2(T) \le \theta_2(L\backslash a) \le \max\{\theta_1(L[C]), \theta_1(L[B_-])\} < \theta_2(T), \]
which is impossible. Therefore $u=a$. Since $h_+\ge 3$ and $p_u\ge 8$, the tree $T$ contains a caterpillar $CP(1,0,8,0,1)$, which is obtained from $P_7$ by adding $8$ leaf neighbors to the middle vertex. Thus
\[\theta_3(T)\le \theta_3(CP(1,0,8,0,1))<\omega.\]
a contradiction.
\qed

We conclude this section with our main theorem.

\begin{theorem}
	Let $T$ be a tree on $n$ vertices, and suppose $T\ne S_n$. Then
	\[\tr(T)\le \tr(DS(n-3,1)),\]
	with equality if and only if
	\[T = DS(n-3,1).\]
\end{theorem}

\section{Beyond connected graphs}
We extend some properties of laziness to disconnected graphs. Let $G_c(n)$ denote the set of graphs on $n$ vertices with $c$ components. Let $\comp{G}_c(n)$ denote the set of graphs whose complements lie in $G_c(n)$. Let $G(\seq{n}{1}{2}{c})$ denote the set of graphs with $c$ components, each of size $n_i$. Let $\comp{G}(\seq{n}{1}{2}{c})$ denote the set of graphs whose complements lie in $G(\seq{n}{1}{2}{c})$.
The following result shows that complementation preserves the ordering of graphs in these classes provided $c\ge 2$.

\begin{lemma}\label{relation_i}
	Let $X$ be any graph. Let $c$ be any integer with $2\le c\le n$. Let $\seq{n}{1}{2}{c}$ be positive integers that sum to $n$.
	\begin{enumerate}[(i)]
		\item $X$ is an $i$-th laziest graph in $G(\seq{n}{1}{2}{c})$ if and only if $\comp{X}$ is an $i$-th laziest graph in $\comp{G}(\seq{n}{1}{2}{c})$.
		\item $X$ is an $i$-th laziest graph in $G_c(n)$ if and only if $\comp{X}$ is an $i$-th laziest graph in $\comp{G}_c(n)$.
		\item $X$ is an $i$-th laziest graph in $G_1(n)\cap \comp{G}_c(n)$ if and only if $\comp{X}$ is an $i$-th laziest graph in $\comp{G}_c(n)$.
	\end{enumerate}
\end{lemma}
\proof
Let $X$ and $Y$ be two graphs in $G(\seq{n}{1}{2}{c})$. Since $c\ge 2$, both $\comp{X}$ and $\comp{Y}$ are connected. Hence by Lemma \ref{complement},
\[\tr(X) - \tr(\comp{X}) = \frac{2}{n}(c-1)= \tr(Y) - \tr(\comp{Y}),\]
from which (i) and (ii) follow. For (iii), note that when $c \ge 2$, we have $X\in G_1(n)\cap \comp{G}_c(n)$ if and only if $\comp{X}\in \comp{G}_c(n)$.
\qed

A direct calculation shows that among all laziest graphs in $G(n, m)$, the more unbalanced $(n, m)$, the lazier the graph.

\begin{lemma}\label{G(n,m)}
	For $n\ge m\ge 2$,
	\[\tr(K_n) + \tr(K_m) < \tr(K_{n+1}) + \tr(K_{m-1}).\]
\end{lemma}

It is shown in \cite{Godsil2023} that the laziest graph in $G(\seq{n}{1}{2}{c})$ is the disjoint union of complete graphs $K_{n_1}\cup K_{n_2} \cup \cdots K_{n_c}$. Using the above results, we find the laziest graph in the other three classes.

\begin{theorem}\label{4laziest}
	The following statements are true.
   \begin{enumerate}[(i)]
	\item \label{lazeist_comp_fact_1} The laziest graph in $G(\seq{n}{1}{2}{c})$ is the disjoint union of complete graphs
	\[K_{n_1} \cup K_{n_2} \cup \cdots \cup K_{n_c},\]
	with trace
	\[n-2c+\frac{2}{n_1} + \cdots + \frac{2}{n_c}.\]
	\item \label{lazeist_comp_fact_2}For $c\ge 2$, the laziest graph in $\comp{G}(\seq{n}{1}{2}{c})$ is the complete multipartite graph
	\[K_{n_1,n_2,\cdots, n_c},\]
	with trace
	\[n-2c+\frac{2}{n_1} + \cdots + \frac{2}{n_c}-\frac{2}{n}(c-1).\]
	\item \label{lazeist_comp_fact_3}The laziest graph in $G_c(n)$ is the disjoint union of complete graphs
	\[K_{n-c+1}\cup (c-1)K_1,\]
	with trace
	\[n-2+\frac{2}{n-c+1}.\]
	\item \label{lazeist_comp_fact_4}For $c\ge 2$, the laziest graph in $\comp{G}_c(n)$ is the complete multipartite graph
	\[K_{n-c+1,1,\cdots, 1},\]
	with trace
	\[n-2+\frac{2}{n-c+1}-\frac{2}{n}(c-1).\]
\end{enumerate}
\end{theorem}
\proof
Statements (ii) and (iv) follow from (i) and (iii), respectively, by taking the complements. Statement (iii) follows from Lemma \ref{G(n,m)}.
\qed

This lets us compare the laziest graphs in $G_c(n)$ and $\comp{G}_c(n)$ with different number of components.

\begin{corollary}\label{ordering}
	 The following statements are true.
	\begin{enumerate}[(i)] 
		\item Among the laziest of $G_c(n)$, we have
		\[\tr(K_n) < \tr(K_{n-1}\cup K_1) < \cdots < \tr(n K_1).\]
		\item Among the laziest of $\comp{G}_c(n)$ where $c\ge 2$, we have
		\[\tr(K_{n-1, 1})> \tr(K_{n-2,1,1}) > \cdots > \tr(K_{\lceil \sqrt{n}\rceil, 1,\cdots, 1}),\] and
		\[\tr(K_{\lfloor \sqrt{n}\rfloor, 1, \cdots, 1})< \cdots < \tr(K_{2,1,\cdots, 1}) < \tr(K_n).\]
	\end{enumerate}
\end{corollary}
\proof
Statement (i) follows from direct comparison. Statement (ii) follows by taking the complements and derivative relative to $c$.
\qed

We now show that the third laziest connected graph on $n$ vertices is $K_{n-2,1,1}$. To start, we cite a collection of results on connected graphs with high Laplacian eigenvalue multiplicity.

\begin{theorem}\cite{Das2007}\label{n-2}
	Let $X$ be a connected graph on $n$ vertices. If some Laplacian eigenvalue of $X$ has multiplicity $n-2$, then $\comp{X}$ is disconnected.
\end{theorem}

\begin{theorem}\cite{Mohammadian2011, Yin2017}\label{n-3}
	Let $X$ be a connected graph on $n$ vertices. If some Laplacian eigenvalue of $X$ has multiplicity $n-3$, then either $\comp{X}$ is disconnected, or $X=C_5$, or $X=G_{n,r}$, which is obtained from the disjoint union of two copies of $K_r + (n/2-r)K_1$ by joining every vertex in one copy of $(n/2-r)K_1$ to every vertex in the other copy of $(n/2-r)K_1$.
\end{theorem}

\begin{theorem}
	Let $X$ be a connected graph on $n$ vertices, and suppose $X\notin\{K_n, S_n\}$. Then
	\[\tr(X) \le \tr(K_{n-2,1,1}),\]
	with equality if and only if $X = K_{n-2,1,1}$.
\end{theorem}
\proof
Let $c$ be the number of components in $\comp{X}$. Since $X$ is non-complete, $c\le n-1$. Suppose first that $c\ge 3$, then by Theorem \ref{4laziest} (iv) and Corollary \ref{ordering} (ii),
\begin{align*}
	\tr(X)& \le n-2 + \frac{2}{n-c+1} - \frac{2(c-1)}{n}\\
	&=\max\{\tr(K_{n-2,1,1}), \tr(K_{2,1,\cdots,1})\}\\
	&=\tr(K_{n-2,1,1}).
\end{align*}
Now suppose $c=2$. Then $X$ is a join, and we may write
\[X = \comp{Y \cup Z}\]
for some graphs $Y$ and $Z$. If $Y=K_1$, then $\comp{X} = K_1\cup Z$ where $Z$ is connected non-complete graph on $n-1$ vertices. By Lemma \ref{complement} and Theorem \ref{star},
\[\tr(X) = \tr(Z) + 1 - \frac{2}{n} \le \tr(S_{n-1}) + 1 - \frac{2}{n} < \tr(K_{n-2,1,1}).\]
If $Y\ne K_1$, let $a=\abs{V(Y)}$ and $b=\abs{V(Z)}$. By Lemma \ref{complement}, Lemma \ref{G(n,m)} and Theorem \ref{4laziest} (i),
\begin{align*}
	\tr(X)&=\tr(Y)+ \tr(Z) - \frac{2}{n} \\
	&\le \tr(K_a) + \tr(K_b) - \frac{2}{n} \\
	&\le \tr(K_2) + \tr(K_{n-2}) - \frac{2}{n}\\
	&<\tr(K_{n-2,1,1}).
\end{align*}

It remains to consider the case where $c=1$. As the statement can be verified computationally for $n\le 9$, we assume $n\ge 10$. Let
\[L = \sum_{\lambda} \lambda E_{\lambda}\]
be the Laplacian spectral decomposition of $X$. If there is a subset $S$ of eigenvalues whose multiplicities sum to $k$ where $3\le k\le n-4$, then a similar argument to the proof of Theorem \ref{star} shows 
\begin{align*}
	\tr(X) &\le \tr(E_0^{\circ 2}) + \tr\left(\left(\sum_{\lambda\in S}E_{\lambda}\right)^{\circ 2}\right) + \tr\left(\left(I-E_0 -\sum_{\lambda\in S}E_{\lambda}\right)\right)^{\circ 2})\\
	&\le \frac{(n-1)^2}{n}-\frac{2k(n-1-k)}{n^2} \\
	&\le \frac{(n-1)^2}{n}- \frac{6(n-4)}{n^2}\\
	&<\tr(K_{n-2,1,1}).
\end{align*}
Thus, the multiplicities of the eigenvalues of $L$ are one of the following:
\[(1,1,n-2), \quad (1,2,n-3),\quad (1,1,1,n-3).\]
Since $\comp{X}$ is connected and $n\ge 10$, by  Theorem \ref{n-2} and Theorem \ref{n-3}, $X=G_{n,r}$ for some even $n$ and some $r$ with $1\le r \le n/2$. Let $s = n/2-r$. Note that $G_{n,r}$ has an equitable partition $\{C_1, C_2, C_3, C_4\}$ with quotient matrix
\[Q = \pmat{s & -s & 0 & 0\\
	-r & r+s & -s & 0\\
	0& -s & r+s & -r\\
	0& 0& -s & s}.\]
It follows that $n/2$ is an eigenvalue of $L$, and
\[(E_{n/2})_{uu} =
\begin{cases}
	1 - \frac{1}{r} + \frac{s}{rn}, &  \text{if } u\in C_1\cup C_4,\\
	1 - \frac{1}{s} + \frac{r}{sn}, & \text{if } u\in C_2 \cup C_3.
\end{cases}\]
Therefore, for $n\ge 10$,
\begin{align*}
	\tr(X)&\le \tr(E_{n/2}^{\circ 2}) + \tr((I - E_{n/2})^{\circ 2})\\
	&=n-6+\frac{10}{n} + \frac{1}{r} + \frac{1}{s}\\
	&\le n-5+\frac{10}{n} + \frac{2}{n-2}\\
	&<\tr(K_{n-2,1,1}). \tag*{\sqr53}
\end{align*}

\section{Future work}
This paper established machinery to study the laziness of a graph. While we have determined the laziest members of some families of graphs, many questions remain. For example, data on graphs with up to $9$ vertices raises the following:

\begin{enumerate}[(i)]
	\item Is $DS(n-3,1)$ the laziest connected graph on $n$ vertices with connected complements?
	\item Is $P_n$ one of the graph that minimize $\tr(X)$ among all graphs on $n$ vertices?
\end{enumerate}
We notice that paths and their complements are not the only graphs attaining such minimum. This motivates the next question.

\begin{enumerate}[(i)]
	\setcounter{enumi}{2}
	\item What can we say about graphs on the same number of vertices with equal laziness?
\end{enumerate}

Our definition of laziness takes an unweighted sum of $\AMM_{uu}$ over all vertices $u$. It will be interesting, and perhaps more practical, to extend this definition to a weighted version (say, by the degrees) of laziness. 

\begin{enumerate}[(i)]
	\setcounter{enumi}{3}
	\item Study the weighted laziness
	\[\tr(X) = \sum_u w_u \AMM_{uu},\quad \text{where } w_u \ge 0,\quad \sum_u w_u=1.\]
\end{enumerate}

Finally, we can consider similar problems for quantum walks relative to other weighted adjacency matrices. It is conjectured in \cite{Godsil2023} that $K_n$ maximizes $\tr(X)$ relative to the adjacency matrix over all connected graphs on $n$ vertices. Data on graphs with up to $9$ vertices also raises the following:
\begin{enumerate}[(i)]
	\setcounter{enumi}{4}
\item  Is $S_n$ the second laziest connected graph on $n$ vertices for the adjacency matrix, unsigned Laplacian matrix and normalized Laplacian matrix?
\end{enumerate}

One may also study the laziness in a more general setting:

\begin{enumerate}[(i)]
	\setcounter{enumi}{5}
	\item Let $A$ and $D$ be the adjacency matrix and degree matrix of $X$, respectively. Given $\gamma \in \re\cup \{\infty\}$, let
	\[\tr(X,\gamma) = \sum_{\lambda} E_{\lambda} \circ E_{\lambda},\]
	where $E_{\lambda}$ is the orthogonal projection onto the $\lambda$-eigenspace of $ A + \gamma D$. Given a family of graphs, determine the critical values of $\gamma$ where the extremizers of $\tr(X, \gamma)$ change.
\end{enumerate}

\section*{Acknowledgement}
This material is based upon work supported by the National Science Foundation under Grant No.  ExpandQISE OSI-2427020 and CCF-2348399.

		\bibliographystyle{amsplain}
		\bibliography{L_amm}
		
	\end{document}